\documentclass[11pt]{article}
\usepackage{graphicx}
\usepackage{amsmath, amsthm, amssymb, amsfonts, latexsym}
\usepackage{array}
\usepackage{caption}
\usepackage{relsize}
\usepackage{subfigure}

\usepackage{fullpage}
\usepackage{afterpage}

\usepackage{epsfig}
\usepackage{fancybox}
\usepackage{array}
\usepackage{shadow}

\usepackage{mathrsfs}
\usepackage{longtable}
\usepackage{calc}
\usepackage{multirow}
\usepackage{hhline}
\usepackage{ifthen}

\numberwithin{equation}{section}
\numberwithin{table}{section}
\numberwithin{figure}{section}

\newtheorem{definition}{Definition}
\newtheorem{theorem}{Theorem}
\newtheorem{lemma}{Lemma}
\newtheorem{remark}{Remark}
\newtheorem{condition}{Condition}
\newtheorem{assumption}{Assumption}

\def\Df{\mathop{D}}
\def\Ds{\mathop{D^2}}

\newcommand{\eqn}{\begin{eqnarray}}
\newcommand{\eqnd}{\end{eqnarray}}

\newcommand{\fpd}[2]{\frac{\partial #1}{\partial #2}}
\newcommand{\spd}[2]{\frac{\partial^2 #1}{\partial {#2}^2}}

\newcommand{\vf}[1]{{\bf#1}} 

\newcommand{\R}{\mathbb{R}}

\newcommand{\argmin}[1]{\rm{argmin}}

\def\harf{\hbox{$\frac{1}{2}$}}

\begin{document}


\title{
Piecewise Constant Policy Approximations to Hamilton-Jacobi-Bellman Equations
}
\author{
C.~Reisinger\thanks{Mathematical Institute, Andrew Wiles Building, University of Oxford, Woodstock Road, Oxford, OX2 6GG, 
{\tt reisinge@maths.ox.ac.uk}}
~and P.A.~Forsyth\thanks{Cheriton School of Computer Science, University of Waterloo, Waterloo ON, Canada N2L 3G1, {\tt paforsyt@uwaterloo.ca}}
}

\maketitle

\begin{abstract}
An advantageous feature of  piecewise constant policy timestepping
for  Hamilton-Jacobi-Bellman (HJB) equations is that 
different linear approximation schemes, and indeed different 
meshes, can be used for the resulting linear equations 
for different control parameters.
Standard convergence analysis suggests that monotone (i.e., linear)
interpolation must be used to transfer data between meshes.
Using the equivalence to a switching system and an adaptation of the usual 
arguments based on consistency, stability and monotonicity,
we show that if limited, potentially higher order 
interpolation is used for the mesh transfer, convergence is guaranteed.
We provide numerical tests for the mean-variance optimal 
investment problem and the uncertain volatility option 
pricing model, and compare the results to published test cases.
\end{abstract}



\section{Introduction}

This article is concerned with the numerical approximation 
of fully nonlinear second order partial differential equations of the form
\begin{eqnarray}
0 \; = \; F({\bf x},V,DV,D^2V) \!\!&=\!\!&
\left\{
\begin{array}{rl}
V_{\tau} - \sup_{q\in Q} L_q V, &~~ {\bf x} \in \R^{d} \times 
                                                    (0,T], \\
                             
                              V( {\bf x} ) - \mathcal{G}( {\bf x} ), & ~~
                                       {\bf x} \in \R^{d} \times
                                       \{0 \},
\end{array}
\right.
\label{hjbintro}
\end{eqnarray}
where ${\bf x}=(S,\tau) $ contains 
both `spatial' coordinates $S\in \R^{d}$ and 
{\em backwards time} $\tau$. For fixed $q$ in a control set $Q$, $L_q$ is the linear differential operator
\begin{eqnarray}
\label{linop}
L_q V = \text{tr} \big(\sigma_q \sigma_q^T D^2 V\big) + \mu_q^T D V - r_q V + f_q,
\end{eqnarray}
where $\sigma_q \in \R^{d \times d}$, $\mu_q \in \R^{d}$ and $r_q, f_q \in \R$ are functions of the control as well as possibly of $\bf x$.
An initial (in backwards time) condition $V(0, \cdot) = \mathcal{G}(\cdot)$ is also specified.

These equations arise naturally from stochastic optimization problems. By \emph{dynamic programming}, the value function satisfies
an HJB equation of the form (\ref{hjbintro}).  Since dynamic programming works
backwards in time from a terminal time $T$ to today $t=0$, it is conventional to
write PDE (\ref{hjbintro}) in terms of backwards time $\tau = T-t$, with $T$ being the terminal
time, and $t$ being forward time.

Many examples of equations of the type (\ref{hjbintro}) are found in financial mathematics, 
including the following:
optimal investment problems \cite{merton69};
transaction cost problems \cite{davis1990};
optimal trade execution problems \cite{almgren2001};
values of American options \cite{karatzas88};
models for financial derivatives under uncertain volatilities \cite{avellaneda95, lyons95};
utility indifference pricing of financial derivatives \cite{carmona09}.
More recently, enhanced oversight of the financial system has resulted
in reporting requirements which include Credit Value Adjustment (CVA)
and Funding Value Adjustment (FVA), which lead to nonlinear  control problems of the form (\ref{hjbintro})
\cite{burgard2011,mercurio2013,burgard2013}.

If the solution has sufficient regularity, specifically for Cordes coefficients, it has recently been demonstrated that
higher order discontinuous Galerkin solutions are possible \cite{smears2014}.
Generally, however, these problems have solutions only in the viscosity sense of \cite{usersguide}.

A general framework for the convergence analysis of 
discretization schemes for strongly nonlinear degenerate elliptic 
equations of type (\ref{hjbintro}) is introduced in \cite{barlessouganidis}, and has since been refined
to give error bounds and convergence orders, see, e.g., \cite{barlesjakobsen02,barlesjakobsen05,barlesjakobsen07}.
The key requirements that ensure convergence are consistency, stability and monotonicity of the discretization.

The standard approach to solve (\ref{hjbintro}) by finite 
difference schemes is to ``discretize, then optimize'', i.e., 
to discretize the derivatives in (\ref{linop})
and to solve the resulting finite-dimensional control problem.
The nonlinear discretized equations are then often
solved using variants of policy iteration \cite{forsythlabahn},
also known as Howard's algorithm and equivalent to Newton's
iteration under common conditions \cite{bokanowski}.

At each step of policy iteration, it is necessary to find
the globally optimal policy (control) at each computational node.
The PDE coefficients may be sufficiently complicated functions of the control
variable $q$ such that the global optimum cannot be found either analytically
or by standard optimization algorithms.
Then, often the only way to guarantee convergence
of the algorithm is to discretize the admissible control
set and determine the optimal control at each node by
exhaustive search, i.e., $Q$ is approximated by 
finite subset
$Q_H=\{q_1,\ldots q_J\} \subset Q$.
This step is the most computationally
time intensive part of the entire algorithm.
Convergence to
the exact solution is obtained by refining $Q_H$.

Of course, in many practical problems, the admissible 
set is known to be of {\em bang-bang} type, i.e., the
optimal controls are a finite subset of the admissible
set.  Then the true admissible set is already
a discrete set of the form $Q_H$.

In both cases, if we use
backward Euler timestepping, an approximation to $V^{n+1}$ at time $\tau^{n+1}$ is obtained from
\begin{eqnarray}
\frac{V^{n+1}-V^n}{\Delta \tau} - \max_{q_j \in Q_H} L^h_{q_j} V^{n+1}  = 0,
\label{pde_problem_1}
\end{eqnarray}
where we have a spatial discretization $L^h_{q_j}$,
with $h$  a mesh size and $\Delta \tau$ the timestep.


\subsection{Objectives}
It is our experience that many industrial practitioners find
it difficult and time consuming to implement a solution of 
equation (\ref{pde_problem_1}).  As pointed out in
\cite{pooley2002}, many plausible discretization
schemes for HJB equations can generate incorrect
solutions. Ensuring that the discrete
equations are monotone, especially if accurate {\em central
differencing as much as possible} schemes are used, is non-trivial \cite{forsythwang08}.
Policy iteration is known to
converge when the underlying discretization operator for a fixed
control is monotone (i.e., an M-matrix) \cite{bokanowski}.
Seemingly innocent approximations may violate the M-matrix condition,
and cause the policy iteration 
to fail to converge.

A convergent iterative scheme for a finite element approximation with quasi-optimal convergence rate to the solution of a strictly elliptic switching system is proposed and analysed in \cite{boulbrachene2001finite}.
Here, we are concerned with parabolic equations and exploit the fact that approximations of the continuous-time control processes by those piecewise constant in time and attaining only a discrete set of values, lead to accurate approximations of the value function.

A technique which seems to be not commonly used (at least in the
finance community) is based on piecewise constant policy time stepping (PCPT) \cite{krylov00,barlesjakobsen07}.
In this method,  given a 
discrete control set $Q_H=\{q_1,\ldots q_J\}$, $J$ independent
PDEs are solved at each timestep.  Each of the $J$ PDEs has
a constant control $q_j$.  At the end of the timestep,
the maximum value at each computational node is determined,
and this value is the initial value for all $J$ PDEs at
the next step.

Convergence of an approximation in the timestep has been analyzed 
in \cite{krylov99} using purely probabilistic techniques,
which shows that under mild regularity  assumptions a 
convergence order of $1/6$ in the timestep can be proven.
In this and other works \cite{krylov97, krylov00}, applications to 
fully discrete schemes are given and their convergence is deduced.
These estimates seem somewhat pessimistic, in that we typically
observe (experimentally) first order convergence.

Note that this technique has the following advantages:
\begin{itemize}
   \item No policy iteration is required.
    \item Each of the $J$ PDEs has a constant policy, and hence
          it is straightforward to guarantee a monotone,
          unconditionally stable discretization.
             \item Since the PCPT algorithm reduces the solution of a nonlinear HJB
       equation to the solution of a sequence of linear PDEs (at each timestep),
       followed by a simple $\max$ or $\min$ operation, it is straightforward to
       extend existing (linear) PDE software to handle
       the HJB case.

    \item Each of the $J$ PDEs can be advanced in time independently.  Hence
       this algorithm is an ideal candidate for efficient parallel implementation. 
              \item In the case where we seek the solution of a Hamilton-Jacobi-Bellman-Isaacs (HJBI) PDE of the form
              \begin{eqnarray}
                V_{\tau} - \inf_{p \in P} \sup_{q\in Q} L_{q,p} V \!&=\!& 0,
              \end{eqnarray}
          the {\em discretize and optimize} approach may fail due
          to the fact that policy iteration may not converge in this
          case \cite{Wal1978}.  However, the PCPT technique can be
          easily applied to these problems.
   
\end{itemize}

In view of the advantages of piecewise policy time stepping, it
is natural to consider some generalizations of the basic algorithm.
Since  each of the $J$ independent PDE solves has a different
control parameter, it is clearly advantageous to use a different
mesh for each PDE solution.  This may involve an interpolation
operation between the meshes.  If we restrict attention to
purely monotone schemes, then only linear interpolation can
be used.

However, in \cite{barlesjakobsen07}, it is noted that
the solution of the PDE (\ref{pde_problem_1}) can
be approximated by the solution of a switching system
of PDEs with a small switching cost.  There, 
it is shown that the solution of the switching system
converges to the solution of (\ref{pde_problem_1})
as the switching cost tends to zero.  In \cite{barlesjakobsen07},
the switching system was used as a theoretical tool to
obtain error estimates.

Building on the work in \cite{barlesjakobsen07}, the main results of this paper are:

\begin{itemize}
\item We formulate the PCPT algorithm in terms of the equivalent
     switching system, in contrast to \cite{krylov99}.  We then show that a non-monotone
     interpolation operation between the switching system
     meshes is convergent to the viscosity solution of (\ref{pde_problem_1}).
     The only requirement is that the interpolation operator
     be of {\em positive coefficient} type.
     This permits use of limited high order interpolation or
monotonicity preserving (but not monotone) schemes.

  \item We will include two numerical examples.  The first example
is an uncertain volatility problem \cite{lyons95,avellaneda95}  with a bang-bang control, where we demonstrate
the effectiveness of a higher order (not monotone) interpolation
scheme.   The second example is a continuous time mean-variance asset allocation
problem \cite{forsythwang10}.  In this case, it is difficult to determine
the optimal policy at each node using analytic methods, and we follow the
usual program of discretizing the control and determining the optimal
value by exhaustive search.  We compare the numerical results obtained using
PCPT and a standard policy iteration method.  Comparable accuracy is obtained
for both techniques, with the PCPT method having a considerably smaller
computational complexity.
\end{itemize}

\subsection{Outline}
In order to aid the reader, we provide here an overview of 
the steps we will follow to carry out our analysis.
We will write the PCPT algorithm in the  unconventional form
\begin{eqnarray}
\label{firstpw}
\frac{ V_j^{n+1} - \max_{k=1}^J V_k^{n} }{\Delta \tau} - L^h_{q_j} V^{n+1}_j = 0,
\end{eqnarray}
where the optimization step is carried out at the beginning of the new
timestep, as opposed to the conventional form whereby the optimization
is performed at the end of the old timestep.  
Note that the scheme is a time-implicit discretization 
for each fixed control $q_j$, while the optimization is carried out explicitly.
As discussed above, a decided advantage of this approach is that 
this scheme is unconditionally stable and yet no nonlinear iterations are needed in every timestep.

In order to carry out our analysis, we perform the
following string of approximations:
\begin{eqnarray}
\label{eqnhjb}
\hspace{-1.3 cm} \text{\fbox{HJB equation}} &\hspace{0. cm}& V_{\tau} - \sup_{q\in Q} L_q V = 0 \\
\label{eqndiscrcontr}
\hspace{-1.3 cm}\text{\fbox{Control discretization}}  &\hspace{0. cm}& V_{\tau} - \max_{q_j \in Q_H} L_{q_j} V = 0 \\
\label{eqnswitching}
\hspace{-1.3 cm}\text{\fbox{Switching system}} &\hspace{0. cm}& \min(V_{j,\tau} - L_{q_j} V_j, V_j - \max_{k\neq j}(V_k-c) ) = 0 \\
\hspace{-1.3 cm}\fbox{Discretization} &\hspace{0. cm}& \min \left(  
                                     \frac{V_j^{n+1}-V_j^n}{\Delta \tau}
                            - L^h_{q_j} V^{n+1}_j, V_j^{n+1} - \max_{k\neq j}(\widetilde{V}^{n+1}_{k,(j)}-c) 
                            \right) = 0 \nonumber \\
\label{eqndiscr}
\end{eqnarray}
In the HJB equation (\ref{eqnhjb}), the control parameter 
$q$ is assumed to take values in a compact set $Q$, and for fixed $q$,
$L_q$ is a second order elliptic operator as per (\ref{linop}).
The compact set is discretized by a finite set 
$Q_H=\{q_1,\ldots,q_J\}$, where $H$ is the maximum distance 
between any element in $Q$ and $Q_H$.  Of course, in the case of a {\em bang-bang} control,
the admissible set is already discrete.

The resulting equation (\ref{eqndiscrcontr}) can then be approximated by 
the \emph{switching system} (\ref{eqnswitching}) as in \cite{barlesjakobsen07} 
when the cost $c>0$ of switching between controls $j=1,\ldots,J$ goes to zero.
When $c \rightarrow 0$, every $V_j$ converges to the solution of (\ref{eqndiscrcontr}).
The switching cost is included to guarantee that the  system (\ref{eqnswitching})
satisfies the no-loop condition \cite{ishii_1991_a}, and hence
a comparison property holds.
We then freeze the policies over time intervals of length $\Delta \tau$, 
i.e., restrict the allowable policies to those that assume one of the $q_j$ over such time intervals, and discretize the
PDEs in space and time. 
Here, we use the same timestep $\Delta \tau$ for the, 
say, backward Euler time discretization of the PDE, but generalizations are straightforward.
We provide for the possibility that the PDEs for 
different controls are solved on different meshes, and in that 
case interpolation of the discretized value function $V_k$ for control $q_k$ onto the
$j$-th mesh is needed. We denote by $\widetilde{V}_{k,(j)}$ this interpolant of $V_k$ evaluated on the $j$-th mesh.

The remainder of this article is organized as follows.
We conclude this section by giving standard definitions and assumptions on the equation (\ref{hjbintro}).
Section \ref{sec:discr-pol} shows that the control space can be approximated by a finite set, which prepares the formulation as a switching system.
Section \ref{sec:pw-const} introduces a discretization based on piecewise constant policy timestepping, while
Section \ref{sec:conv} contains the main result proving convergence of these approximation schemes satisfying a standard set of conditions, to
the viscosity solution of a switching system.
Section \ref{sec:num-ex-1d} constructs numerical examples for the mean-variance asset allocation problem and the uncertain volatility option pricing model.
Section \ref{sec:discuss} concludes.

\subsection{Preliminaries}

We now give the standard definition of a viscosity solution before making assumptions on $F$.
Given a function $f:\Omega \rightarrow \mathbb{R}$, where $\Omega \subset \R^n$ open, we first define
the upper semi-continuous envelope as
\begin{eqnarray}
   f^*(\vf{x}) = \lim_{r \rightarrow 0^+} \sup\bigl\{ f(\vf{y}) ~\bigm\vert~
                                  y   \in B(\vf{x},r) \cap   \Omega
                               \bigr\},
\end{eqnarray}
where $B(\vf{x},r) = \{ \vf{y} \in \mathbb{R}^n ~\bigm\vert~ |\vf{x}-\vf{y}| < r\}$.
We also have the obvious definition for a lower semi-continuous
envelope $f_*(\vf{x})$.

%
\begin{definition}[Viscosity Solution]\label{VisSol_def}
A locally bounded function $U: \Omega \rightarrow \mathbb{R}$ is a viscosity
subsolution (respectively supersolution) of \eqref{hjbintro}
if and only if for all smooth test functions $\phi \in C^{\infty}$,
and for all maximum (respectively minimum) points $\bf{x}$ of
$U^* - \phi$ (respectively $U_{*} - \phi$), one has
\begin{eqnarray}
& &  F_{*} \left( \mathbf{x}, U^*(\mathbf{x}), \Df\phi(\mathbf{x}),
          \Ds\phi(\mathbf{x}), U^*( {\bf{x} } ) \right)
  \leq 0 \nonumber \\
\biggl( {\mbox{ respectively}} & & F^* \left (  \mathbf{x}, U_{*}(\mathbf{x}),
               \Df\phi(\mathbf{x}), 
          \Ds\phi(\mathbf{x}), U_{*} ({\bf{x}})
         \right)
  \geq 0 \biggr) ~.
\end{eqnarray}
A locally bounded function $U$ is a viscosity solution if it is both
a viscosity subsolution and a viscosity supersolution.
\end{definition}

\begin{remark}[Smoothness of test functions]
Definition \ref{VisSol_def} specifies that $\phi \in C^{\infty}$,
whereas the common definition uses $\phi \in C^2$.  The equivalence
of these two definitions is discussed in \cite{BarlesNotes1997,Seydel2009}.
Letting $\phi \in C^{\infty}$ simplifies the consistency analysis.
\end{remark}

\begin{assumption}[Properties of $F(\cdot)$]
\label{coeff_assump}
\label{comparison_assump}
We assume that $Q$ is a compact set and that
$\sigma_q \sigma_q^T, \mu_q, r_q, f_q$ are bounded on $\mathbb{R}^{d+1}\times Q$, Lipschitz in $\bf{x}$ uniformly in $q$ (i.e., there is a Lipschitz constant which holds for all $q$) and continuous in $q$.
\end{assumption}

\begin{remark}[Comparison principle]
\label{rem_comp}
Assumption \ref{coeff_assump} is the same as the one made in \cite{barlesjakobsen07}.
It guarantees (see, e.g., \cite{barlesjakobsen07, fleming2006controlled}) that
a strong comparison principle holds for $F$, i.e., 
if $V$ and $W$ are viscosity sub- and supersolutions, respectively, of
(\ref{hjbintro}), with $V(\cdot,0)\le W(\cdot,0)$, then $V\le W$ everywhere.
It also ensures the well-posedness of the switching system (\ref{eqnswitching}), see \cite{ishii_1991_a} and \cite{barlesjakobsen07}.
\end{remark}








\section{Approximation by finite control set}
\label{sec:discr-pol}
In this section, we analyze the validity of the first approximation step from (\ref{eqnhjb}) to (\ref{eqndiscrcontr}), i.e.,
that the compact control set may be approximated by a finite set.
More precisely, for compact $Q \subset{\R}^m$ (i.e., $m\in \mathbb{N}$ is the dimension of the parameter space) and $Q_h\subset Q$ such that
\eqn
\max_{q \in Q} \min_{\hat{q} \in Q_h} |q-\hat{q}| \le h,
\label{denseset}
\eqnd
we define a discrete control HJB equation by
\begin{eqnarray}
0 \; = \; F_h({\bf x},V,DV,D^2V) \!\!&=\!\!&
\left\{
\begin{array}{rl}
V_{\tau} - \sup_{q\in Q_h} L_q V, &~~ {\bf x} \in \R^{d} \times 
                                                    (0,T], \\
                             
                              V( {\bf x} ) - \mathcal{G}( {\bf x} ), & ~~
                                       {\bf x} \in \R^{d} \times
                                       \{0 \}.
\end{array}
\right.
\label{hjbintroh}
\end{eqnarray}

The following lemma will be useful.

\begin{lemma}[Properties of $F(\cdot)$]
\label{lipshitz_assump}
Under Assumption \ref{coeff_assump}, for any $\bf{x}$ and any $C^{\infty}$ test function $\phi$
\begin{eqnarray}
   | F({\bf{x}}, \phi( {\bf{x}} ),  D \phi( {\bf{x}} ), D^2 \phi( {\bf{x}} )  ) 
      - F_h( {\bf{x}}, \phi({\bf{x}}) + \xi ,  D \phi( {\bf{x}} ), D^2 \phi( {\bf{x}} )  )   | 
      & \leq&   \omega_h({\bf{x}},h) + \omega_{\xi}(\xi), \nonumber \\
     \text{where} & & \omega_h({\bf{x}},h) \rightarrow 0 \text{ as } h \rightarrow 0, \nonumber \\
     & & \omega_{\xi}(\xi) \rightarrow 0 \text{ as } \xi \rightarrow 0,
   \label{vis_consistent_a}
\end{eqnarray}
where $\omega_h( {\bf{x}}, h)$ is locally
      Lipshitz continuous in ${\bf{x}}$, independent of $h$.
\end{lemma}

\begin{proof}
See Appendix \ref{lemma1_proof}.
\end{proof}

\begin{theorem}
\label{theorem:discr}
Given Assumption \ref{comparison_assump},
let
$V$ and $V_h$ be the unique viscosity solutions to
\begin{eqnarray}
   F( {\bf{x}}, V, DV, D^2 V) &=& 0, 
        \label{vis_1} \\
   F_h( {\bf{x}}, V_h, D V_h, D^2V_h) & = & 0. 
      \label{vis_2}
\end{eqnarray}
Then
$V_h \rightarrow V$ uniformly on compact sets as $h\rightarrow 0$.
\end{theorem}

\begin{proof}
A consequence of Lemma \ref{lipshitz_assump} is
that
\begin{eqnarray}
   F_h^*({\bf{x}}, \phi({\bf{x}}) + \xi, \dots)  
   &  \leq& F^*( {\bf{x}}, \phi({\bf{x}}), \dots)
    + \omega_h( {\bf{x}}, h)  + \omega_{\xi}( \xi).
    \label{vis_consistent}
\end{eqnarray}

Define, for all $\bf{x}$,
\begin{eqnarray}
   \underline{ V }_h( {\bf{x}}) &=&
               \liminf_{ \substack{
                  { \bf{y}} \rightarrow {\bf{x}}
               }}  {V}_h({\bf{y}}), \\
   \underline{ V }(\bf{x})
     &=& \liminf_{ \substack{ h \rightarrow 0 \\
                  { \bf{y}} \rightarrow {\bf{x}}
               }}  \underline{ V}_h({\bf{y}}).
               \label{defsupersol}
\end{eqnarray}

We claim that $\underline{ V }$ is a viscosity
supersolution of equation (\ref{vis_1}).
To show this, fix $\hat{\bf{x}}$  and choose a smooth test function $\phi$ such that
$\hat{\bf{x}}$ is a global minimum of $\underline{V} - \phi$,
and that 
\begin{eqnarray}
   \underline{V}(\hat{\bf{x}}) = \phi(\hat{\bf{x}} ).
         \label{test_min}
\end{eqnarray}

Then, there exists a sequence $\hat{\bf{x}}_h \rightarrow \hat{\bf{x}}$,
$h \rightarrow 0$, $\underline{V}_h(\hat{\bf{x}}_h) \rightarrow 
\underline{V}( \hat{\bf{x}})$,
such that $\hat{\bf{x}}_h$ is  a global minimum of 
$\underline{V}_h(\hat{\bf{x}}_h) - \phi(\hat{\bf{x}}_h)$.  At
each point $\hat{\bf{x}}_h$, since
$V_h$ is a viscosity solution of
(\ref{vis_2}),
\begin{eqnarray}
 0 \leq F_h^*(\hat{\bf{x}}_h, \underline{V}_h ( \hat{\bf{x}}_h ), 
                 D \phi ( \hat{\bf{x}}_h ),
                 D^2 \phi ( \hat{\bf{x}}_h ) ). \label{vis_3}
\end{eqnarray}
Let $\underline{V}_h(\hat{\bf{x}}_h) = \phi(\hat{\bf{x}}_h) + \xi_h$,
$\xi_h \rightarrow 0 , h \rightarrow 0$,
so that equation (\ref{vis_3}) becomes
\begin{eqnarray}
 0 \leq  F_h^*(\hat{\bf{x}}_h, \phi ( \hat{\bf{x}}_h ) + \xi_h, 
                 D \phi ( \hat{\bf{x}}_h ),
                 D^2 \phi ( \hat{\bf{x}}_h ) ) .
     \label{vis_3a}
\end{eqnarray}
From equations (\ref{vis_consistent}) and (\ref{vis_3a})
we obtain
\begin{eqnarray}
0 
      & \leq & F^*( \hat{\bf{x}}_h, \phi ( \hat{\bf{x}}_h ), 
                 D \phi ( \hat{\bf{x}}_h ),
                 D^2 \phi ( \hat{\bf{x}}_h ) ) + \omega_h( {\bf{x}},h) + \omega_\xi( \xi_h),
\end{eqnarray}
which gives us
\begin{eqnarray}
   0  & \leq & \limsup_{ \substack{ h \rightarrow 0 \\
                   \hat{\bf{x}}_h \rightarrow \hat{\bf{x}}
               }}  
               F^*( \hat{\bf{x}}_h, \phi ( \hat{\bf{x}}_h ), 
                 D \phi ( \hat{\bf{x}}_h ),
                 D^2 \phi ( \hat{\bf{x}}_h ) )
         + \limsup_{ \substack{ h \rightarrow 0 ~;~ \xi_h \rightarrow 0\\
                   { \bf{x}} \rightarrow \hat{\bf{x}}
                    }
             }  \bigl( \omega_h( {\bf{x}},h) + \omega_{\xi}( \xi_h) \bigr)  \nonumber \\
     & \leq & F^*( \hat{\bf{x}}, \phi( \hat{\bf{x}} ), 
                 D \phi ( \hat{\bf{x}} ),
                 D^2 \phi ( \hat{\bf{x}} ) )  \nonumber \\
   & = & F^*( \hat{\bf{x}}, \underline{V}( \hat{\bf{x}} ), 
                 D \phi ( \hat{\bf{x}} ),
                 D^2 \phi ( \hat{\bf{x}} ) ),
\end{eqnarray}
where we use equation  (\ref{test_min})
and Lemma \ref{lipshitz_assump}.

Using similar steps, we can show that $\overline{ V }$ defined similar to (\ref{defsupersol})
is a viscosity subsolution of equation (\ref{vis_1}). Invoking the strong comparison principle,
$\overline{ V } = \underline{ V } = V$.  Uniform convergence on compact sets
           follows using the same argument as in
           Remark 6.4, of \cite{usersguide}.
\end{proof}

\section{Piecewise constant policy timestepping}
\label{sec:pw-const}
Here, we use the equivalence of (\ref{eqndiscrcontr}) and (\ref{eqnswitching}) established in \cite{barlesjakobsen07} as $c\rightarrow 0$ to formulate (\ref{eqndiscr}) precisely.
Consider the HJB equation 
\begin{eqnarray}
    V_{\tau} &=& \max_{ q_j \in Q} {L}_{q_j} V, \nonumber \\
     Q &=& \{ q_1, q_2, \ldots, q_J \},
   \label{hjb_1}
\end{eqnarray}
where we assume a discrete set of controls $Q$.
We have shown in Section \ref{sec:discr-pol} that the optimal value under controls chosen from a compact set can be approximated by a control problem with a finite set.

According to \cite{barlesjakobsen07}, we can also approximate
equation (\ref{hjb_1}) by a {\em switching system}.
Let $U_{j}, j=1,\ldots, J$, be the solution of a
system of HJB equations, with
\begin{eqnarray}
   \min \biggl[  U_{j,\tau} - {L}_{q_j} U_j ,
               U_j - \bigl( \max_{k \neq j} \left( U_k - c \right)
                     \bigr)
           \biggr] &=& 0, ~~{\bf x} \in \R^{d} \times 
                                                    (0,T], \nonumber \\
     U_j - \mathcal{G}( {\bf x} ) &=& 0, ~~ 
                                       {\bf x} \in \R^{d} \times
                                       \{0 \}.
         \label{switch_1}
\end{eqnarray}
The constant $c>0$ is required in order to add some small
transaction 
cost to switching from $j \rightarrow k$. 
This cost term also ensures that only
a finite number of switches can occur (otherwise there would
be an infinite transaction cost; see also Remark \ref{rem:swipen}).
It is shown in \cite{barlesjakobsen07} that $U_j\rightarrow V$ as $c\downarrow 0$, for all $j$.

For computational purposes, we define a finite computational domain
$\Omega \subset \R^{d}$.  Let $\partial \hat{\Omega}$ denote
the portions of $\partial \Omega$ where we apply approximate
Dirichlet conditions.
We use the usual notation for representing (\ref{switch_1}).
Define
\begin{eqnarray}
  \mathbf{x} = (S, \tau) ~,~~DU = (U_{\tau} ,  U_{S})~,~~D^2 U = U_{SS}~,
\end{eqnarray}
and let $\mathcal{B}_j ( {\bf x})$ be the approximate Dirichlet
boundary conditions on $\partial \hat{\Omega}$. Then, the {\em localized}
problem is defined as 
\begin{eqnarray}
 \label{switch_2}
 0  & = & F_j \left( \mathbf{x}, U_j, \Df U_j, \Ds U_j, \{ U_k \}_{k \neq j} \right) \\
 && ~~~~~~~~~= 
 \left\{
 \begin{array}{rl}
 \min \biggl[  U_{j,\tau} - {L}_{q_j} U_j ,
               U_j - \bigl( \max_{k \neq j} \left( U_k - c \right)
                     \bigr)
                 \biggr], & ~~~{\bf x} \in ~\Omega \backslash \partial \hat{\Omega} \times (0,T], \\
  U_j( {\bf x}) - \mathcal{G}( {\bf x}), & ~~~{\bf x} \in ~ \Omega \times \{0\}, \nonumber \\
  U_j( {\bf x}) - \mathcal{B}_j ( {\bf x}), & ~~~{\bf x} \in ~ \partial \hat{\Omega} \times (0,T],
\end{array}
\right.
\end{eqnarray}
for $j=1,\ldots, J$.
Letting $p_j = \Df U_j, s_j = \Ds U_j$, we can write equation (\ref{switch_2})
as
\begin{eqnarray}
      F_j \left( \mathbf{x}, U_j, p_j, s_j, \{ U_k \}_{k \neq j} \right) = 0.
      \label{switch_2a}
\end{eqnarray} 
Note that system (\ref{switch_2}) is quasi-monotone (see \cite{ishii_1991}),
since
\begin{eqnarray}
     F_j \left( \mathbf{x}, U_j, p_j, s_j, \{ U_k \}_{k \neq j} \right)
     & \leq &
      F_j \left( \mathbf{x}, U_j, p_j, s_j, \{ W_k \}_{k \neq j} \right)
       \quad {\mbox{ if }} ~ U_k \geq W_k;~k \neq j .
    \label{quasi_monotone}
\end{eqnarray}
We include here the definition of a viscosity solution
for systems of PDEs of the form (\ref{switch_2a})
as used in
\cite{ishii_1991,briani_2011,ishii_1991_a}.

\begin{definition}[Viscosity solution of switching system]\label{Vis_def}
A locally bounded function $U: \Omega \rightarrow \mathbb{R}^J$ is a viscosity
subsolution (respectively supersolution) of \eqref{switch_2a}
if and only if for all smooth test functions $\phi_j \in C^{\infty}$,
and for all maximum (respectively minimum) points $\bf{x}$ of
$U^*_j - \phi_j$ (respectively $U_{j*} - \phi_j$), one has
\begin{eqnarray}
& &  F_{j*} \left( \mathbf{x}, U_j^*(\mathbf{x}), \Df\phi_j(\mathbf{x}),
          \Ds\phi_j(\mathbf{x}), \{ U_k^*( {\bf{x} } ) \}_{k \neq j}
             \right)
  \leq 0 \nonumber \\
\biggl( {\mbox{ respectively}} & & F_j^* \left (  \mathbf{x}, U_{j*}(\mathbf{x}),
               \Df\phi_j(\mathbf{x}), 
          \Ds\phi_j(\mathbf{x}), \{ U_{k*} ({\bf{x}}) \}_{k \neq j}
         \right)
  \geq 0 \biggr) ~.
\end{eqnarray}
A locally bounded function $U$ is a viscosity solution if it is both
a viscosity subsolution and a viscosity supersolution.
\end{definition}

\begin{remark}
Note that the $j$-th test function only replaces the derivatives
operating on $U_j$. The terms which are a function of $U_k, k\neq j$, are not
affected.
\end{remark}
We discretize (\ref{switch_2})  using the idea of piecewise constant policy
timestepping.  Define a set of nodes $S_{j,i}$ and timesteps $\tau^n$,
with discretization parameters $h$ and $\Delta \tau$, i.e.,
\begin{eqnarray}
\max_{\stackrel{1\le j\le J}{S\in \Omega}} \min_i  |S - S_{j,i}| & = h, \\
   \max_n (\tau^{n+1} - \tau^n) & = \Delta \tau. \nonumber
\end{eqnarray}
The distinction between $\Delta \tau$ and $h$ is useful for the formulation of the algorithm, but somewhat arbitrary for the analysis. We will therefore label meshes and approximations by $h$ and assume that
\[
\Delta \tau\rightarrow 0 \quad \text{as} \quad h\rightarrow 0.
\]
Define
\begin{eqnarray}
   \mathbf{x}_{j,i}^n(h)  = ( S_{j,i}, \tau^n; h) \in \Omega_{j,h},
\end{eqnarray}
where $(S_{j,i}, \tau^n)$ refer to points on a specific grid $j$,
and the set of grid points on the grid parameterized by
$h$ is $\Omega_{j,h}$.

Then we denote the discrete approximation to $U_j(\mathbf{x}_{j,i}^n)$
on a grid parameterized by $h$  by
$u_j(h, \mathbf{x}_{j,i}^n)$, which is extended to a function $u_j(h,\cdot)$ on $\Omega \times \{\tau_n\}$ by interpolation.
We will sometimes use the shorthand
notation
\begin{eqnarray}
   u_{j,i}^n = u_j(h, \mathbf{x}_{j,i}^n), \quad \mathbf{x}_{j,i}^n =  (S_{j,i} , \tau^n),
\end{eqnarray}
where the dependence on $h$ is understood implicitly.
Note that by parameterizing $ \mathbf{x}_{j,i}^n$ by the control index $j$,
we are allowing for different discretization grids for different controls.

Let $\mathcal{L}_{q_j}^h$ be the discrete
form of the operator $\mathcal{L}_{q_j}$.  We discretize
equation (\ref{switch_2}) for ${\bf x} \in \Omega \backslash \partial \hat{\Omega} \times 
(0,T]$, using $\Delta \tau=\tau^{n+1}-\tau^n$ constant for simplicity, by \emph{piecewise constant policy timestepping}
\begin{eqnarray}
\nonumber
  u_{j,i}^{n+\frac{1}{2}} &=& \max \Bigl[ u_{j,i}^n, \max_{k \neq j} \left( \tilde{u}_{k,i(j)}^n - c \right) 
               \Bigr], \\
  u_{j,i}^{n+1} - \Delta \tau {L}_{q_j}^h u_{j,i}^{n+1}
      & = & u_{j,i}^{n+\frac{1}{2}}, \; \qquad\qquad j=1, \ldots, J, 
    \label{dis_1}
\end{eqnarray}
where
\begin{eqnarray}
       \tilde{u}_{k,i(j)}^n \equiv u_k(h, \mathbf{x}_{j,i}^n)
\end{eqnarray}
is
the value of this interpolant $u_k(h,\cdot)$ of $u_k(h, \mathbf{x}_{j,i}^n)$ at the $i$-th
point of grid $\Omega_{j,h}$.


Discretization (\ref{dis_1}) applies the ``$\max$'' constraint at the
beginning of a new timestep.  Conventionally, one thinks of piecewise constant
policy timestepping as applying the constraint at the end of a timestep.

Clearly, these would be algebraically the same thing if $u_{j,i}^{n+\frac{1}{2}}$ instead of $u_{j,i}^{n}$ was
considered as approximation to $U_j(\mathbf{x}_{j,i}^n)$.
So at the final timestep, these two possible approximations only differ by a final max operation.
However, it will be convenient to apply the constraint as in equation (\ref{dis_1}).
We can then rearrange equation (\ref{dis_1}) to obtain an equation in the form
\begin{eqnarray}
   & &  \hspace{-2.5 cm} G_j\biggl( \mathbf{x}_i^{n+1}, h, u_{j,i}^{n+1}, \{ u_{j,a}^{b+1} \}_{
                                              \substack{a\neq i \\
                                               \text{or }  b\neq n}
                                              },
                       \{ \tilde{u}_{k,i}^n \}_{k \neq j}
              \biggr) \nonumber \\
    &=& 
    \min \biggl[
       \frac{u_{j,i}^{n+1} - u_{j,i}^n }
            {\Delta \tau}
         - {L}_{q_j}^h u_{j,i}^{n+1} ,
        u_{j,i}^{n+1} - \max_{k \neq j} \left( \tilde{u}_{k,i(j)}^n - c \right)
        - \Delta \tau {L}_{q_j}^h u_{j,i}^{n+1}
    \biggr]  \nonumber \\
   &=& 0, ~~ {\bf x} \in \Omega \backslash \partial \hat{\Omega} \times (0,T].
   \label{dis_gen}
\end{eqnarray}
In the event that $L_q$ is strictly elliptic, then we can interpret
the effect of switching at the beginning of the timestep
as adding a vanishing viscosity term 
to the switching part of the equation.  However, note that we do not
in general
require that $L_q$ be strictly elliptic.

We omit the trivial discretizations of $F_j(\cdot)$ in the remaining (boundary) portions
of the computational domain.
Note that the notation
\begin{eqnarray}
    \{ u_{j,a}^{b+1} \}_{
                                              \substack{a\neq i \\
                                               \text{or }  b\neq n}
                                              }
\end{eqnarray}
refers to the set of discrete solution values at nodes neighbouring (in time and space) node $(i,n+1)$.

\section{Convergence of approximations to the switching system}
\label{sec:conv}
Here, we prove convergence of the piecewise constant policy 
approximation (\ref{eqndiscr}) to the solution of the 
switching system (\ref{eqnswitching}). We start by summarizing the main conditions
we need.  We will verify that these conditions are satisfied for all the
schemes we use in our numerical examples.

\begin{condition}[Positive interpolation]
\label{assumption:interpolation}
We require that the interpolant $\tilde{u}_{k,i(j)}^n$ of the $k$-th grid onto the $i$-th point of the $j$-th grid can be written as
\begin{eqnarray}
\label{interp}
    \tilde{u}_{k,i(j)}^n &=& \sum_{ \alpha \in N^k(j,i,n) } \omega_{k,i(j),\alpha}^n u_{k,\alpha}^n 
      =  \sum_{\alpha \in N^k(j,i,n)  } \omega_{k,i(j),\alpha}^n u_k( h, \mathbf{x}_{k,\alpha}^n ), 
\end{eqnarray}      
where
\begin{eqnarray}
\sum_{\alpha \in N^k(j,i,n) } \omega_{k,i(j),\alpha}^n = 1, \qquad \omega_{k,i(j),\alpha}^n \geq 0,
            \label{interp_wts}
\end{eqnarray}
and $N^k(j,i,n) $ are the neighbours to the point $\mathbf{x}_{j,i}^{n}$ on 
grid $\Omega_{k,h}$.  
\end{condition}

\begin{remark}[Monotone versus limited interpolation]
Condition (\ref{interp}), (\ref{interp_wts}) is clearly satisfied by linear interpolation on simplices and multi-linear interpolation on hyper-rectangles,
in which case the weights $\omega_{k,i(j),\alpha}^n$ are only functions of the coordinates $x^n_{j,i}$ and $x^n_{k,\alpha}$. This interpolation is then also a monotone operation and results in an overall monotone scheme, as defined below in Condition \ref{assumption:monotonicity}.

We can, however, relax the requirement of monotonicity of the interpolation step by allowing weights
$\omega_{k,i(j),\alpha}^n = \omega_{k,i(j),\alpha}^n(u_k^n)$, i.e., possibly nonlinear functions of the interpolated nodal values.
It is easy to see that as long as
\begin{eqnarray}
\label{limited_interp}
\min_{\alpha \in N^k(j,i,n)} u_{k,\alpha}^n  \;\; \le  \;\;  \tilde{u}_{k,i(j)}^n \;\; \le \;\; \max_{\alpha \in N^k(j,i,n)} u_{k,\alpha}^n,
\end{eqnarray}
one can then always find weights such that (\ref{interp}), (\ref{interp_wts}) hold.
One can enforce (\ref{limited_interp}) by a simple limiting step applied to any, potentially higher order, interpolant.

Generally, the decomposition (\ref{interp}) will not be the way by which a particular interpolation is originally defined or constructed in practice. The point is that whenever (\ref{limited_interp}) holds, weights of this form exist and this is all that is needed for the theoretical analysis.
An example are one-dimensional monotonicity preserving interpolants such as those in \cite{fritsch}.  As discussed in
\cite{debrabantjakobsen12}, the interpolation in \cite{fritsch} 
is constructed to be monotonicity preserving (i.e., the interpolant is increasing over intervals where the interpolated values are increasing), and therefore satisfies condition (\ref{interp_wts}). 
\end{remark}

\begin{condition}[Weak Monotonicity]
\label{assumption:monotonicity}
We require that discretization (\ref{dis_gen}) be monotone with respect to
$u_{j,a}^{b+1}, \tilde{u}_{k,i}^n$, i.e., if
\begin{eqnarray}
   w_{j,i}^{n} &\geq& u_{j,i}^{n} \qquad \forall (i,j,n),
           \nonumber \\
    \tilde{w}_{k,i(j)}^{n} &\geq& \tilde{u}_{k,i(j)}^{n} \qquad \forall (i,k,n),
    \label{monotone_1}
\end{eqnarray}
then
\begin{eqnarray}
 G_j\biggl( \mathbf{x}_i^{n+1}, h, u_{j,i}^{n+1}, \{ w_{j,a}^{b+1} \}_{
                                              \substack{a\neq i \\
                                               \text{or }  b\neq n}
                                              },
                       \{ \tilde{w}_{k,i(j)}^n \}_{k \neq j}
              \biggr)  
  \leq G_j\biggl( \mathbf{x}_i^{n+1}, h, u_{j,i}^{n+1}, \{ u_{j,a}^{b+1} \}_{
                                              \substack{a\neq i \\
                                               \text{or }  b\neq n}
                                              },
                       \{ \tilde{u}_{k,i(j)}^n \}_{k \neq j}
              \biggr)~.
   \label{mono_1}
\end{eqnarray}
\end{condition}

\begin{remark}
Note that the requirement that the scheme be monotone in $\tilde{u}_{k,i}^n$ (the interpolated solution) is a weaker
condition than requiring monotonicity in $u_{k,\alpha}^n$. In particular, let us re-iterate that we are not requiring interpolation to be
a monotone operation, as long as Condition \ref{assumption:interpolation} is satisfied.
\end{remark}

\begin{condition}[$l_\infty$ stability]
\label{assumption:stability}
We require that the solution of 
equation (\ref{dis_gen}), $u_j(h, \mathbf{x}_{j,i}^{n+1})$,  exists and
is bounded independent of $h$.
\end{condition}

\begin{remark}
The bounds (\ref{limited_interp}) implied by Condition \ref{assumption:interpolation} ensure that the interpolation step does not increase
the $l_\infty$ norm of the solution.
\end{remark}

\begin{condition}[Consistency]
\label{assumption:consistency}
We require that we have  local consistency, in the sense that,
for any smooth function $\phi_j$, and any functions 
$\rho_k$ (not necessarily smooth)
\begin{eqnarray}
      & & \hspace{-1.5 cm} \biggl|  G_j\biggl( \mathbf{x}_{j,i}^{n+1}, h, 
                \phi_{j,i}^{n+1} + \xi, 
                  \{   \phi_{j,a}^{b +1} \}_{
                                              \substack{a\neq i \\
                                               \text{or }  b\neq n}
                                              } + \xi,
                       \{ \tilde{\rho}_k(\mathbf{x}_{j,i}^{n_\ell}  ) \}_{k \neq j}
              \biggr) \biggr. \nonumber \\
    & & \biggl.  - F_j \biggl( \mathbf{x}_{j,i}^{n+1}, \phi_j( \mathbf{x}_{j,i}^{n+1} ),
             \Df\phi_j( \mathbf{x}_{j,i}^{n+1} ),
             \Ds\phi_j( \mathbf{x}_{j,i}^{n+1} ),
             \{ \tilde{\rho}_{k}( \mathbf{x}_{j,i}^{n_\ell}) \}_{k \neq j} \biggr) \biggr|
         \leq  \omega_1(h) +  \omega_2 (\xi),  \nonumber \\
       & & \qquad \qquad \qquad ~~~~~~~~~~~~~~~~~~~~~~~~~~~~ \omega_1(h) \rightarrow 0 \text{ as } h \rightarrow 0, \quad \omega_2(\xi) \rightarrow 0 \text{ as } \xi \rightarrow 0.
              \label{sub_sol_2}
\end{eqnarray} 
\end{condition}

\begin{theorem}
\label{theorem:gridconv}
Under Assumption \ref{coeff_assump}, 
the solution of any scheme of the form (\ref{dis_gen})
satisfying Conditions \ref{coeff_assump}--\ref{assumption:consistency}
converges to the viscosity solution of (\ref{switch_2}), uniformly on bounded domains.
\end{theorem}
\begin{proof}
We basically follow along the lines in \cite{barlessouganidis}, with the 
generalizations in \cite{briani_2011} for weakly coupled systems.  We include
the details here in order to show that we only require
Condition \ref{assumption:interpolation} for the interpolation operator,
and this permits use of certain classes of high order interpolation.

Define the upper semi-continuous function $\overline{u}$ by
\begin{eqnarray}
  \overline{u}_j( \hat{\mathbf{x}} )
   = \limsup_{ \substack{ h \rightarrow 0 \\
                          \mathbf{x}_{j,i}^{n+1} \rightarrow \hat{\mathbf{x}}
                         }
             }  ~ u_{j}(h, \mathbf{x}_{j,i}^{n+1} )~,
\end{eqnarray}
where $\mathbf{x}_{j,i}^{n+1} \in \Omega_{j,h}.$
Similarly, we define the lower semi-continuous function $\underline{u}$ by
\begin{eqnarray}
  \underline{u}_j( \hat{\mathbf{x}} )
   = \liminf_{ \substack{ h \rightarrow 0 \\
                          \mathbf{x}_{j,i}^{n+1} \rightarrow \hat{\mathbf{x}}}
             }  ~ u_{j}(h, \mathbf{x}_i^{n+1} )~.
\end{eqnarray}
Note that the above definitions imply that $\overline{u}_j^* = \overline{u}_j$
and $\underline{u}_{j*} = \underline{u}_j$.

Let $\hat{\mathbf{x}}$ be fixed and $\phi_j$ be a smooth test function such that
\begin{eqnarray}   
   \phi_j( \hat{\mathbf{x}}) &= &\overline{u}_j(  \hat{\mathbf{x}} ), \nonumber \\
   \phi_j( \mathbf{x}) & > &\overline{u}_j(  \mathbf{x} ), \quad \mathbf{x} \neq \hat{\mathbf{x}} .
      \label{phi_prop}
\end{eqnarray}
This of course means that $ (\overline{u}_j - \phi_j)$ has a global maximum at 
$\mathbf{x} = \hat{\mathbf{x}}$.
Consider a sequence of
grids with discretization parameter $h_\ell$, such that
$h_\ell \rightarrow 0$ for $l \rightarrow \infty$.
We use the notation
\begin{eqnarray}
   \mathbf{x}_{j,i_l}^{n_\ell+1} = (S_{j,i_\ell}, \tau^{n_\ell+1}; h_\ell)
          \in \Omega_{j,h_\ell}
\end{eqnarray}
to refer to the grid point $(i_\ell, n_\ell+1)$ on the grid parameterized
by $h_\ell$, associated with discrete control $q_j$.  
Let $\mathbf{x}_{j,i_\ell}^{n_\ell+1}$ be the point on grid $\Omega_{j,h_\ell}$
such that
\begin{eqnarray}
   \biggl( u_j(h_\ell, \mathbf{x}_{j,i_\ell}^{n_\ell+1}) - \phi_j(\mathbf{x}_{j,i_\ell}^{n_\ell+1})
   \biggr)
     \label{glob_max}
\end{eqnarray}
has a global maximum, where $\phi_j$ is a test function
satisfying equation (\ref{phi_prop}).   Note that in general,
for any finite $h_{\ell}$, $\mathbf{x}_{j,i_\ell}^{n_\ell+1} \neq \hat{\mathbf{x}}$.

Following the usual arguments from \cite{barlessouganidis},
and more particularly in \cite{briani_2011}, for
a sequence of grids $\Omega_{j,h_\ell}$ parameterized by $h_l$, there exists
a set of grid nodes $(i_\ell, n_\ell+1)$, such that
(\ref{glob_max}) is a global maximum and
\begin{eqnarray}
   & \mathbf{x}_{j,i_\ell}^{n_\ell+1} \rightarrow \hat{\mathbf{x}},~~
   \mathbf{x}_{j,i_\ell}^{n_\ell} \rightarrow \hat{\mathbf{x}} ,~~
   u_j(h_\ell, \mathbf{x}_{j,i_\ell}^{n_\ell+1}) \rightarrow \overline{u}_j( \hat{\mathbf{x}}), \quad
   \text{for }
   \ell \rightarrow \infty, &  \nonumber
\end{eqnarray}
where $\mathbf{x}_{j,i_\ell}^{n_\ell+1}, \mathbf{x}_{j,i_\ell}^{n_\ell} \in \Omega_{j,h_\ell}$
and, for $k \neq j$, noting equation (\ref{interp_wts}) 
for the interpolant $\tilde{u}_{k}(\mathbf{x}_{k,i_{\ell}(j)}^{n_\ell})$,
we have
\begin{eqnarray}
 \tilde{u}_{k,i_{\ell}(j)}^{n_\ell}  \equiv \tilde{u}_{k}( h_\ell, \mathbf{x}_{j,i_\ell}^{n_\ell})
         = \sum_{\alpha_\ell \in N^k(j,i_\ell,n_\ell) } 
                   \omega_{k,i_{\ell}(j), \alpha_\ell}^{n_\ell} u_{k, \alpha_\ell}^{n_\ell}
    = \sum_{\alpha_\ell \in N^k(\cdot)} \omega_{k,i_\ell, \alpha_\ell}^{n_\ell} u_{k}( \mathbf{x}_{k, \alpha_\ell}^{n_\ell} ),
      && k \neq j~, \nonumber \\
   \sum_{\alpha_\ell \in N^k(\cdot)} \omega_{k,i_\ell, \alpha_\ell}^{n_\ell}
      \mathbf{x}_{k,\alpha_\ell}^{n_\ell} \rightarrow \hat{\mathbf{x}}, &&
    \ell \rightarrow \infty~ ,   \nonumber \\
     \displaystyle \limsup_{ \ell  \rightarrow \infty} 
       \tilde{u}_{k, i_{\ell}(j)} \leq \overline{u}_k( \hat{\mathbf{x}})~, &&
    \label{interp_wts_2}
\end{eqnarray}
where $ \mathbf{x}_{k,\alpha_\ell}^{n_\ell} \in \Omega_{k,h_\ell}$ and $\hat{\mathbf{x}} = (\hat{S}, \hat{\tau})$.
Let 
\begin{eqnarray}
    u_{j,i_\ell}^{n_\ell+1} &\equiv& u_j(h_\ell, \mathbf{x}_{j,i_\ell}^{n_\ell+1})~, 
                \nonumber \\
     \phi_{j,i_{\ell}}^{n_\ell+1} & \equiv & \phi( \mathbf{x}_{j,i_\ell}^{n_\ell+1}) 
  ~.
\end{eqnarray}
Note that for any finite mesh size $h_\ell$ the global maximum in equation (\ref{glob_max})
is not necessarily zero, hence we define $\xi_l$ by
\begin{eqnarray}
   u_{j,i_\ell}^{n_\ell+1} &=& \phi_{j,i_{\ell}}^{n_\ell+1} + \xi_l~, \nonumber
\end{eqnarray}
such that
\begin{eqnarray}
  \xi_{\ell} \rightarrow 0 \quad \text{for } \ell \rightarrow \infty~.
     \label{xi_1}
\end{eqnarray}
Since (\ref{glob_max}) is a global maximum at $\mathbf{x}_{i_\ell}^{n_\ell+1}$, then
\begin{eqnarray}
  \{ u_{j,a_\ell}^{b_\ell+1} \}_{
                          \substack{a_\ell\neq i_\ell \\
                          \text{or }  b_\ell\neq n_\ell}
                       } 
   &\leq&  \{ \phi_{j,a_\ell}^{b_\ell +1} \}_{
                     \substack{a_\ell\neq i_\ell \\
                          \text{or }  b_\ell\neq n_\ell}
                       } 
              + \xi_{\ell}~.
             \label{xi_2}
\end{eqnarray}

Substituting equations (\ref{xi_1}) and (\ref{xi_2}) into 
equation (\ref{dis_gen}), and using the monotonicity
property of the discretization (\ref{monotone_1})  gives
\begin{eqnarray}
    0 \geq 
         G_j\biggl( \mathbf{x}_{j, i_\ell}^{n_\ell+1}, h_\ell, 
                \phi_{j,i_{\ell}}^{n_\ell+1} + \xi_\ell, 
                  \{   \phi_{j,a_\ell}^{b_\ell +1} \}_{
                                              \substack{a\neq i \\
                                               \text{or }  b\neq n}
                                              } + \xi_\ell,
                       \{ \tilde{u}^{n_\ell}_{k,i_\ell} \}_{k \neq j}
              \biggr) ~.
   \label{sub_sol_1a}
\end{eqnarray}
Note that we do not replace $\{ \tilde{u}_{k,i_\ell}^{n_\ell} \}_{k \neq j}$ by the test function
in  equation (\ref{sub_sol_1a}). This is because the test function is only defined
such that $ \phi_j \geq \overline{u}_j$,  and there is no such relationship with
$u_k, k \neq j$.
Let
\begin{eqnarray}
   \tilde{\rho}_{k, i_\ell}^{n_\ell} = \tilde{\rho}_{k}( \mathbf{x}_{j,i_\ell}^{n_\ell} )  
        = \max( \tilde{u}_k( \mathbf{x}_{j,i_\ell}^{n_\ell}), 
                        \overline{u}_k( \hat{  \mathbf{x} } ))~.
              \label{upper_rho}
\end{eqnarray}
From equation (\ref{interp_wts_2}), we have that
\begin{eqnarray}
   \lim_{\ell \rightarrow \infty}  \{ \tilde{\rho}_{k}( \mathbf{x}_{j,i_\ell}^{n_\ell}) \}_{k \neq j}
    = \overline{u}_k( \hat{\mathbf{x}} )~,
    \label{limit_uk_a}
\end{eqnarray}
and that
\begin{eqnarray}
\tilde{\rho}_{k, i_\ell}^{n_\ell} \geq \tilde{u}^{n_\ell}_{k,i_\ell}~.\label{limit_uk_aa}
\end{eqnarray}
Substituting equation (\ref{upper_rho})  into
equation (\ref{sub_sol_1a}), and using the monotonicity
property of the discretization (\ref{monotone_1})  gives
\begin{eqnarray}
    0 \geq 
         G_j\biggl( \mathbf{x}_{j, i_\ell}^{n_\ell+1}, h_\ell, 
                \phi_{j,i_{\ell}}^{n_\ell+1} + \xi_\ell, 
                  \{   \phi_{j,a_\ell}^{b_\ell +1} \}_{
                                              \substack{a\neq i \\
                                               \text{or }  b\neq n}
                                              } + \xi_\ell,
                       \{ \tilde{\rho}^{n_\ell}_{k,i_\ell} \}_{k \neq j}
              \biggr) ~.
   \label{sub_sol_1}
\end{eqnarray}

Equations (\ref{sub_sol_1}) and (\ref{sub_sol_2}) then imply that 
\begin{eqnarray*}
   0 & \geq & F_j \biggl( \mathbf{x}_{j,i_\ell}^{n_\ell+1}, \phi_j( \mathbf{x}_{j,i_\ell}^{n_\ell+1} ),
             \Df\phi_j( \mathbf{x}_{j,i_\ell}^{n_\ell+1} ),
             \Ds\phi_j( \mathbf{x}_{j,i_\ell}^{n_\ell+1} ),
             \{ \tilde{\rho}_{k}( \mathbf{x}_{j,i}^{n_\ell}) \}_{k \neq j} \biggr) 
      - \omega_1( h_\ell) - \omega_2( \xi_\ell )~.
\end{eqnarray*}
Recalling equation (\ref{limit_uk_a}), we have that
\begin{eqnarray}
   0  & \geq & \liminf_{\ell \rightarrow \infty}  F_j
           \biggl( \mathbf{x}_{i_\ell}^{n_\ell+1}, \phi_j( \mathbf{x}_{i_\ell}^{n_\ell+1} ),
             \Df\phi_j( \mathbf{x}_{i_\ell}^{n_\ell+1} ),
             \Ds\phi_j( \mathbf{x}_{i_\ell}^{n_\ell+1} ),
          \{ \tilde{\rho}_{k}( \mathbf{x}_{j,i_\ell}^{n_\ell}) \}_{k \neq j}
         \biggr)
                 \nonumber \\
      & & -  \liminf_{\ell \rightarrow\infty}  \omega_1(h_\ell) 
           - \liminf_{\ell \rightarrow \infty} \omega_2( \xi_\ell) \nonumber \\
    & \geq & 
       F_{j*} \biggl( \hat{\mathbf{x}}, \phi_j( \hat{\mathbf{x}} ),
                  \Df\phi_j(  \hat{\mathbf{x}} ), \Ds\phi_j(  \hat{\mathbf{x}} ),
                            \{ \overline{u}_k(  \hat{\mathbf{x}} ) \}_{k \neq j} \biggr) \nonumber \\
    & = & F_{j*} \biggl( \hat{\mathbf{x}}, \overline{u}_j(\hat{\mathbf{x}} ),
                     \Df\phi_j(  \hat{\mathbf{x}} ), \Ds\phi_j(  \hat{\mathbf{x}} ),
                           \{\overline{u}_k(  \hat{\mathbf{x}} ) \}_{k \neq j} \biggr)~.
\end{eqnarray}
Hence $\overline{u}_j$ is a subsolution of equation (\ref{switch_2}).
A similar argument shows that $\underline{u}_j$ is a supersolution of
equation (\ref{switch_2}).  Since a strong comparison principle
holds for the switching system under Assumption \ref{coeff_assump} (see Proposition 2.1 in \cite{barlesjakobsen07} and Remark \ref{rem_comp}), we then have $\underline{u}_j = \overline{u}_j$ is the unique continuous viscosity solution of equation (\ref{switch_2}).

We have thus shown the result.
\end{proof}

\begin{remark}
Note that we have $\{\tilde{u}_{k,i(j)}^n\}$ appearing in equation (\ref{dis_gen}),
which would appear to cause a problem in terms of consistency,
since we cannot assume that $\tilde{u}_{k,i(j)}^n = \tilde{u}_{k,i(j)}^{n+1} + O(h)$ since
$u_{k,\alpha}^n$ are not necessarily smooth.  The key fact here is that since
equation (\ref{limit_uk_a}) holds, we do not need smoothness.
\end{remark}

\begin{remark}[Interpolation and switching cost]
\label{rem:swipen}
We include here a brief example of the role of the switching cost $c$ in (\ref{switch_1}).
Assume we solve the degenerate equation $u_t = 0$ and write it as the trivial HJB equation
$\sup_{\theta\in \{0,1\}} u^\theta_t = 0$.
We represent $u^0$ and $u^1$ on two different meshes with nodes at
$x^\theta_i = (i+0.5 \, \theta) h$, i.e., shifted by half a mesh size. The fully implicit discretization over
a single timestep is $u^{\theta,n+1} = u^{\theta,n}$  for both controls.
Consider now the situation of convex $(u_i^0)$, such that linear interpolation increases the solution.
Using piecewise linear
interpolation with $c=0$ at the end of the timestep,
\[
u^{1,n+1}_i = (u^{0,n+1}_i + u^{0,n+1}_{i+1})/2,
\]
and similarly for $u^{0,n+1}_i$.
Repeating this for the next timestep, the piecewise constant policy discretization is equivalent to
\[
u^{\theta,n+2}_i = \mathsmaller{\frac{1}{4}} \, u^{\theta,n}_{i-1} + \mathsmaller{\frac{1}{2}} \, u^{\theta,n}_{i} + \mathsmaller{\frac{1}{4}} \, u^{\theta,n}_{i+1}
\quad \Leftrightarrow \quad
\frac{u^{\theta,n+2}_i-u^{\theta,n}_i}{(h/2)^2} = \frac{u^{\theta,n}_{i-1} - 2 u^{\theta,n}_{i} + u^{\theta,n}_{i+1}}{h^2}.
\]
If we pick $\Delta \tau = h^2/8$, this discretization is consistent with the standard heat equation $u_t = u_{xx}$ instead of $u_t = 0$.

With the cost $c>0$ switched on, for sufficiently small $h$, we will have
\[
u^{1,n}_i > (u^{0,n+1}_i + u^{0,n+1}_{i+1})/2 - c,
\]
such that $u^{1,n+1} = u^{1,n}$ and, by induction, the equation $u_t=0$ is solved exactly.

For an overall convergent method, one has to pick $h$ and $\Delta t$ as a function of $c$, 
depending on the interpolation method and smoothness of the solution,
and then let $c \rightarrow 0$.
We discuss this in more detail in Section \ref{sec:num} for concrete examples.

\end{remark}

\begin{remark}[Reduction to standard piecewise constant policy method \cite{krylov99,barlesjakobsen07}]
\label{c_0_error}
Discretization (\ref{dis_1}), with $c=0$, can be viewed as a form of the usual
piecewise constant policy method \cite{krylov99,barlesjakobsen07}. As a result,
if linear interpolation is used to transfer information between grids,
then (\ref{dis_1}) is a monotone discretization of HJB equation (\ref{hjb_1}),
which is easily shown to satisfy the standard requirements for convergence
to the viscosity solution.  
If we use standard finite difference schemes, we would
expect the spatial error (for smooth test functions)
to be of size $O(h)$ with timestepping error of size $O(\Delta \tau)$.
In this case ($c=0)$, if the solution is smooth, we would
expect the total discretization error to be of size
$O(h) + O(\Delta \tau) + O( h^2/\Delta \tau) $, where the last term
arises from a linear interpolation error accumulated $O(1/\Delta \tau)$ times.
Note that this term will be absent in the case $c>0$, since the finite
switching cost will prevent an $O(1/\Delta \tau)$ accumulation
of interpolation error.
\end{remark}

\section{Numerical examples}
\label{sec:num}

In this section, we study the convergence of discretization schemes based on piecewise constant policy timestepping in numerical experiments.
We present two examples, the uncertain volatility model from derivative pricing, and a mean-variance asset allocation problem.
In both examples, we investigate the convergence with respect to the timestep and mesh size.

\label{sec:num-ex-1d}

\subsection{The uncertain volatility model}

We study first the uncertain volatility option pricing model \cite{lyons95}. 
In this example, we  also examine the role of the switching cost
and the impact of different interpolation methods.

The super-replication value of a European-style derivative 
is given by the HJB equation
\begin{equation}
\label{uncertainpde1d}
\fpd{V}{\tau} - \sup_{\sigma \in \Sigma} 
{L}_{\sigma} V = 0,
\end{equation}
where
\begin{equation}
\label{eqnLS}
{L}_{\sigma} V =
\harf  \sigma^2 S^2 \spd{V}{S} 
+ r S \fpd{V}{S}
 - r V
\end{equation}
for $S\in (0,\infty)$, $ \tau \in (0,T]$, with $\tau = T - t$ and
\begin{equation}
\Sigma = [\sigma_{\min},\sigma_{\max}].
\end{equation}
In addition to the PDE, the value function satisfies a terminal condition. 
For the numerical tests, we choose a butterfly payoff function $P$ such that
\[
V(S,0) =  P(S) = \max(S-K_1,0)-2 \max(S-K,0) + \max(S-K_2,0)
\]
and we localize the domain to $[S_{\min},S_{\max}]$ with
\begin{eqnarray*}
   V(S_{\max}, \tau) &= &0 , ~~~~~~(S,\tau) \in \{S_{\max} \} \times (0,T] ~,
              \nonumber \\
   V_{\tau} & = & -rV, ~~(S,\tau) \in \{S_{\min} \} \times (0,T],
\end{eqnarray*}
and parameters as in Table \ref{tab:paramuv}.
\begin{table}
\centering
\begin{tabular}{|c|c|c|c|c|c|c|c|}
\hline 
$r$ & $\sigma_{\min}$ & $\sigma_{\max}$ & $T$ & $K$ & $K_1$ & $K_2$ & $S_0$\\
\hline
$0.05$ & $0.3$ & $0.5$ & $1$ & $100$ & $80$ & $120$ & $100$ \\
\hline
\end{tabular}
\caption{Model parameters used in numerical experiments for uncertain volatility model.}
\label{tab:paramuv}
\end{table}


It is well understood \cite{lyons95} that the optimal control 
is always attained at one of the interval boundaries (`bang-bang'),
depending on the sign of the second derivative of the value function $V$.
The payoff function was chosen such that $V$ has mixed convexity and therefore the optimal control differs between different regions of the state space and changes in time.

We use logarithmic coordinates $X=\log S$, so instead of (\ref{eqnLS}) we approximate
\begin{equation}
\label{eqnLX}
{L}_{\sigma} V =
\harf  \sigma^2 \spd{V}{X} 
+ \left(r-\sigma^2/2\right) \fpd{V}{X}
 - r V,
\end{equation}
and now the coefficients are bounded and $X\in [\log S_{\min}, \log S_{\max}]$,
so Assumption \ref{coeff_assump} is satisfied. 
Here, it is straightforward to construct ``positive coefficient'' schemes for the linear PDE arising when the control set is a singleton. 
A well-established route allows us to construct monotone, consistent and $\ell_{\infty}$
stable 
schemes for the fully non-linear problem from positive coefficient discretization
operators $L^h_q$ using a direct control method of the ``discretize, then optimize'' type, see \cite{forsythlabahn},
\begin{eqnarray}
\label{directcontr}
\min_{j\in\{1,2\}}
\left(  \frac{V^{n+1}-V^n}{\Delta \tau}
- L^h_{q_j} V^{n+1}
\right) = 0,
\end{eqnarray}
where
$q_1 = \sigma_{\min}$ and $q_2 =\sigma_{\max}$.

The resulting non-linear finite dimensional system of equations can be solved by policy iteration (Howard's algorithm, see \cite{bokanowski}).
In particular, we will use standard central finite differences in space and an implicit Euler discretization in time. For small enough $h$, this scheme
is monotone, $\ell_{\infty}$ stable and consistent in the viscosity sense.

We compare the direct control method to several variants of piecewise constant policy timestepping on the basis of these discretizations, as described in Section \ref{sec:pw-const}, specifically equation (\ref{dis_1}). 
Here, only two control values have to be considered and the switching system is two-dimensional. 
In the case of no interpolation,
the scheme simplifies to
\begin{eqnarray}
\label{uvpw}
\frac{ V_j^{n+1} - \max_{k\in \{1,2\} } \left(V_k^{n} - c_{k,j} \right) }{\Delta t}
 & = &  L^h_{q_j} V_j^{n+1} , \qquad j = 1,2, \nonumber \\
&& \qquad c_{k,j} = \begin{cases}
                   c, & k \neq j, \\
                   0, & k = j.
                 \end{cases}
\end{eqnarray}
We again discretize $L_{q_j}$ using a positive coefficient discretization, hence
it is straightforward to verify that Assumption \ref{coeff_assump} and Conditions \ref{assumption:interpolation}--\ref{assumption:consistency} 
hold (on the localized domain $ S \in [ S_{\min},  S_{\max}]$).

A solution extrapolated from the finest meshes was computed as an approximation to the exact solution and used to estimate the errors.
The numerical value of this solution is $ V(S_0,0) = 1.67012$ (see also \cite{pooley}). 
From this, we approximate the error as
\begin{eqnarray}
\label{uverror}
e(h,\Delta \tau) = \vert
V(S_0,0) - \widetilde{U}_1(h,\Delta \tau,c; S_0,0)
\vert,
\end{eqnarray}
where $V$ is the exact solution and $\widetilde{U}_1(h,\Delta \tau, c; \cdot, \cdot)$ the numerical approximation
to $U_1$, the first component of the switching system, for mesh size $h$, timestep $\Delta \tau$,
and switching cost $c$.

\subsubsection*{Dependence on timestep $\Delta \tau$ and switching cost $c$}

We first analyze how the switching cost affects convergence of the approximations.
In Fig.~\ref{fig:costplots}, we compare the following two cases.
\begin{figure}
\includegraphics[width=0.495\columnwidth]{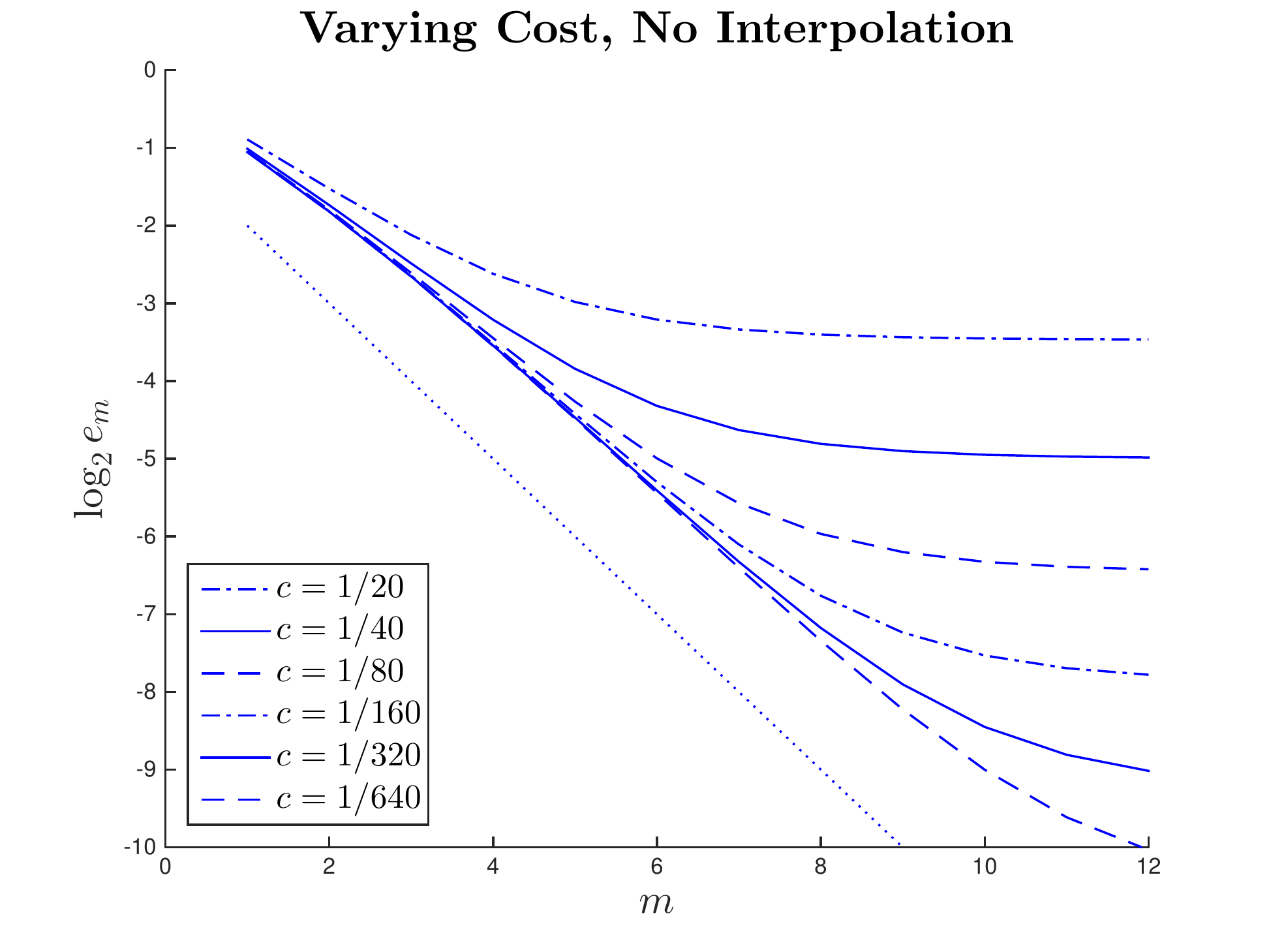}  \hfill
\includegraphics[width=0.495\columnwidth]{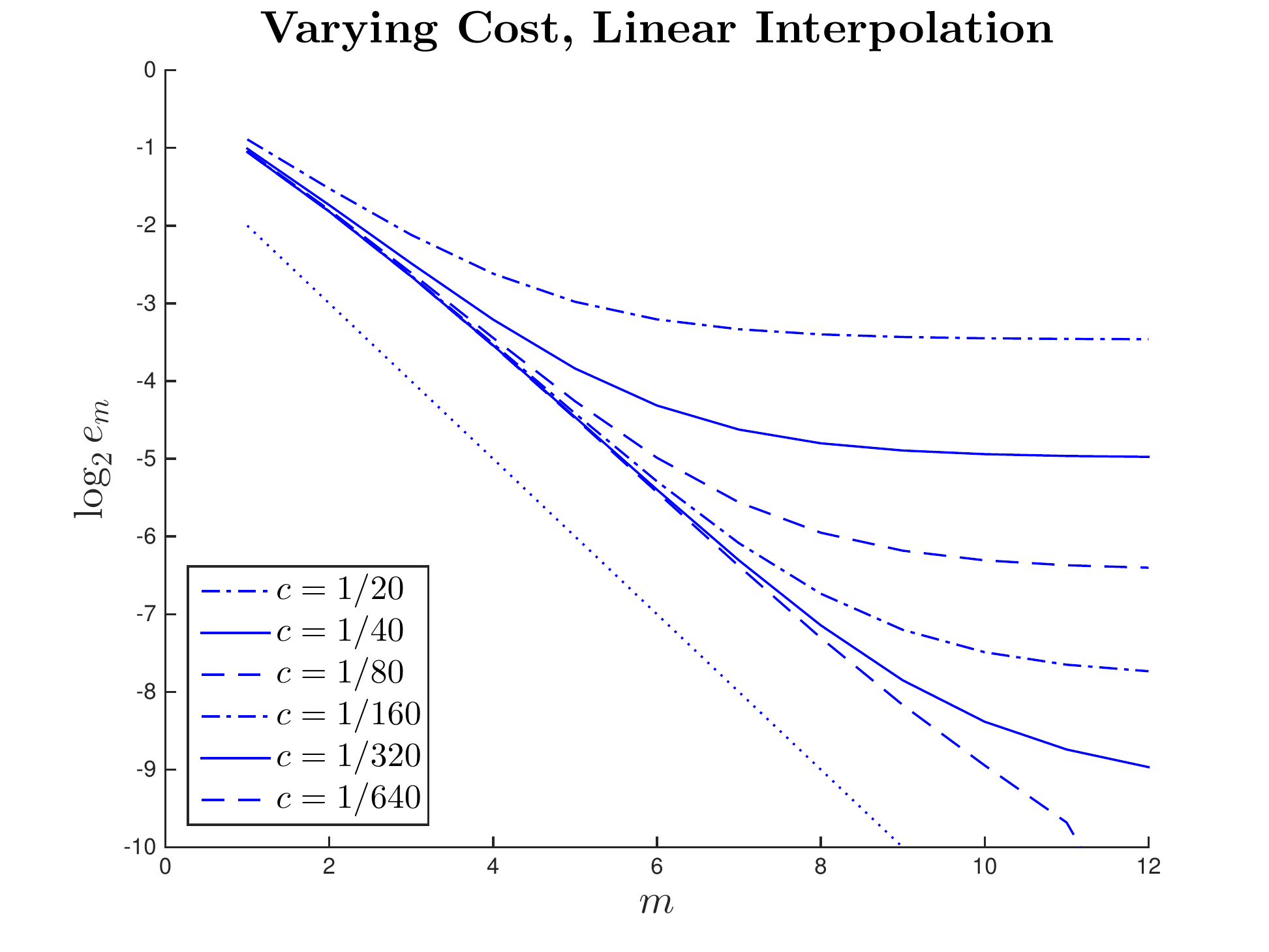} \\
\caption{The uncertain volatility test case with parameters as in Table \ref{tab:paramuv}. Shown is for different methods the $\log_2$ error,
where $e_m = e((\Delta \tau)_m,h,c)$ from (\ref{uverror}) is the error
for timestep $(\Delta \tau)_m=1/8 \cdot 2^{-m} \in \{1/8,\ldots, 1/32768\}$ and
in each plot, from top to bottom, $c=1/20,1/40,1/80,1/160,1/320,1/640$. The mesh size is fixed at $h=1/1024$.
The dotted line has slope $-1$.
Left: Piecewise constant policy timestepping on a single mesh; right: linear interpolation between individual meshes for each control.}
\label{fig:costplots}
\end{figure}

\begin{enumerate}
\item
{\bf Policy timestepping, fixed mesh}:
We use (\ref{uvpw}) on a single uniform mesh on $[\log(K) - 4\cdot \bar{\sigma},
\log(K) + 4\cdot \bar{\sigma}]$, i.e., encompassing four standard deviations either side, where $\bar{\sigma}$ is the average of the two extreme volatilities. 

For fixed switching cost, the error first decreases linearly as the timestep decreases, but eventually converges to a non-zero value.
This convergence is monotone. 
Moreover, for fixed (sufficiently small) $h$, the asymptotic error for $\Delta \tau\rightarrow 0$ is decreasing with $c$.

\item
{\bf Policy timestepping, linear interpolation}:
We also study the use of separate meshes for the two components of the switching system, ${V}_{1}$ and ${V}_{2}$. 
In particular, we use uniform meshes on the intervals $[\log(K) - 4\cdot {\sigma_{\min}}, \log(K) + 4\cdot \sigma_{\min}]$ and
$[\log(K) - 4\cdot {\sigma_{\max}}, \log(K) + 4\cdot \sigma_{\max}]$, so that for the chosen parameters mesh points on the two meshes do not coincide.
Then, interpolation 
is necessary to represent these solutions on both meshes and to evaluate the explicit terms 
(those at time-level $n+1$) in (\ref{uvpw}).
In the case of linear interpolation, 
the overall scheme is monotone, and trivially satisfies Condition \ref{assumption:interpolation}. 
As a result of the switching cost,
the cumulative effect of linear interpolation in each timestep
is controlled even for small $\Delta \tau$ (see Remark \ref{c_0_error}).
\end{enumerate}

Next, we analyze the convergence jointly in $h$ and $\Delta \tau$ for fixed $c$, as well 
as the convergence with respect to $c$ for the case with interpolation.  The results are given in 
Table \ref{tab:uv}.
For fixed positive switching cost $c>0$, we compute approximations on a sequence of time and space meshes with $N_k = 2 N_{k-1}$
and $M_k=2 M_{k-1}$.
The asymptotic ratio of about two
is consistent with an error of $O(h) + O(\Delta \tau)$.

We now study the difference between the solution of the switching system for 
fixed cost $c$ and the solution of the HJB equation (\ref{uncertainpde1d}).
Considering the boldface values in the table as good approximations for this difference, we observe convergence 
as $c\rightarrow 0$ which is roughly consistent with order $3/4$. The theoretically proven order of $1/3$ from Theorem 2.3 in \cite{barlesjakobsen07} does not seem sharp for these data.

The approximation error in $\Delta t$ and $h$ does not appear to be affected by $c$ as long as $c>0$, which is seen by comparison of the lines `(c)' in Table \ref{tab:uv} for $c=1/10$ to $c=1/640$, for small enough mesh parameters (last three columns).
In computations, it would therefore seem prudent to pick $\Delta t = O(h)$
and $c = O(h^{4/3})$ i.e., proportional errors,
even though faster, but more erratic convergence is obtained setting $c=0$ uniformly.

\begin{table}
\centering
{\small
\begin{tabular}{|c|c|c|c|c|c|c|c|c|c|}
\hline
$N_k$&  & $32$ & $64$ & $128$ & $256$ & $512$ & $1024$ & $2048$ & $4096$ \\
$M_k$ &  & $512$ & $1024$ & $2048$ & $4096$ & $8192$ & $16384$ & $32768$ & $65536$ \\
\hline \hline
$c$ &&&&&&&&& \\
\hline
$1/10$ & (a) & 2.0692  &  1.9724   &  1.9660   & 1.9532  &   1.9491  &  1.9478   & 1.9474  &  1.9472 \\ 
&(b)& 0.3991  &  0.3023  &  0.2958  &  0.2831  &  0.2789  &  0.2777   & 0.2773  &  {\bf 0.2771} \\
& (c) &&  -0.0968  & -0.0065  & -0.0127  & -0.0042  & -0.0012 &  -0.0004 &  -0.0002 \\
& (d) &&& 14.9669  &  0.5079  &   3.0621  &  3.4257  &  2.8297  &  2.3076 \\ \hline
$1/40$ & (a) & 1.6126  &  1.5441  &  1.7406  &  1.7663  &  1.7623  &  1.7612   & 1.7608  &  1.7606 \\
&(b)& -0.0575 &  -0.1260   & 0.0705 &   0.0961  &  0.0922  &  0.0910  &  0.0906   & {\bf 0.0905} \\
&(c)&& -0.0685  &  0.1965  &  0.0256 &  -0.0040 &  -0.0011 &  -0.0004  & -0.0002 \\
&(d)&&& -0.3486  &  7.6692  & -6.4629  &  3.5076  &  2.7965  &  2.3037  \\ \hline
$1/160$ &(a)& 1.1989  &  1.2205  &  1.5510  &  1.7019  &  1.7033  &  1.7022  &  1.7018  &  1.7016 \\
&(b)& -0.4712 &  -0.4496 &  -0.1191  &  0.0318  &  0.0331  &  0.0320  &  0.0317  &  {\bf 0.0315} \\
&(c)&& 0.0216  &  0.3305  &  0.1509  &  0.0013  & -0.0011 &  -0.0004 &  -0.0002 \\
&(d)&&& 0.0653  &  2.1902 & 112.7836  & -1.2152  &  2.7969  &  2.2994  \\ \hline
$1/640$ &(a)& 0.9256  &  0.9897  &  1.3752   & 1.6448  &  1.6833  &  1.6822  &  1.6818  &  1.6816 \\
&(b)& -0.7446 &  -0.6804 &  -0.2949  & -0.0254  &  0.0132  &  0.0121  &  0.0117  & {\bf 0.0115} \\
&(c)&& 0.0641  &  0.3855  &  0.2696  &  0.0385 &  -0.0011  & -0.0004 &  -0.0002 \\
&(d)&&& 0.1664  &  1.4301  &  6.9932  & -35.1108  &  2.7835  &  2.3009   \\ \hline
$0$ &(a)& 0.8271  &  0.7221  &  1.0227  &  1.4109   & 1.6654  & 1.6702 &   1.6703  &  1.6702 \\
&(b)& -0.8430 &  -0.9480 &  -0.6474  & -0.2592 &  -0.0047  &  0.0001  &  0.0002  &  0.0001 \\
&(c)&& -0.1050  &  0.3005  &  0.3882  &  0.2545  &  0.0048  &  0.0001 &  -0.0001 \\
 &(d)&&& -0.3494  &  0.7742  &  1.5252  & 53.3568  & 39.0446 & -1.6580 \\ \hline
\end{tabular}
}
\caption{The uncertain volatility test case with parameters as in Table \ref{tab:paramuv};
piecewise constant policy timestepping with linear interpolation between individual meshes for each control;
convergence with respect to mesh parameters and switching cost. Shown are:
(a) the numerical solution $V_k=\widetilde{V}(N_k,M_k,c;S_0,0)$;
(b) the difference to the exact solution $V_k-V(S_0,0)$;
(c) the increments $V_k-V_{k-1}$;
(d) the ratios of increments $(V_{k}-V_{k-1})/(V_{k-1}-V_{k-2})$.}
\label{tab:uv}
\end{table}

\subsubsection*{Dependence on timestep $\Delta \tau$ and mesh size $h$}

In Fig.~\ref{fig:meshcostplots}, we show the convergence in both the timestep and mesh size for two different costs,
$c=0.01$ and $c=0.16$. Compare this to the case $c=0$ in Fig.~\ref{fig:1dplots} (middle row, left).
\begin{figure}
\includegraphics[width=0.495\columnwidth]{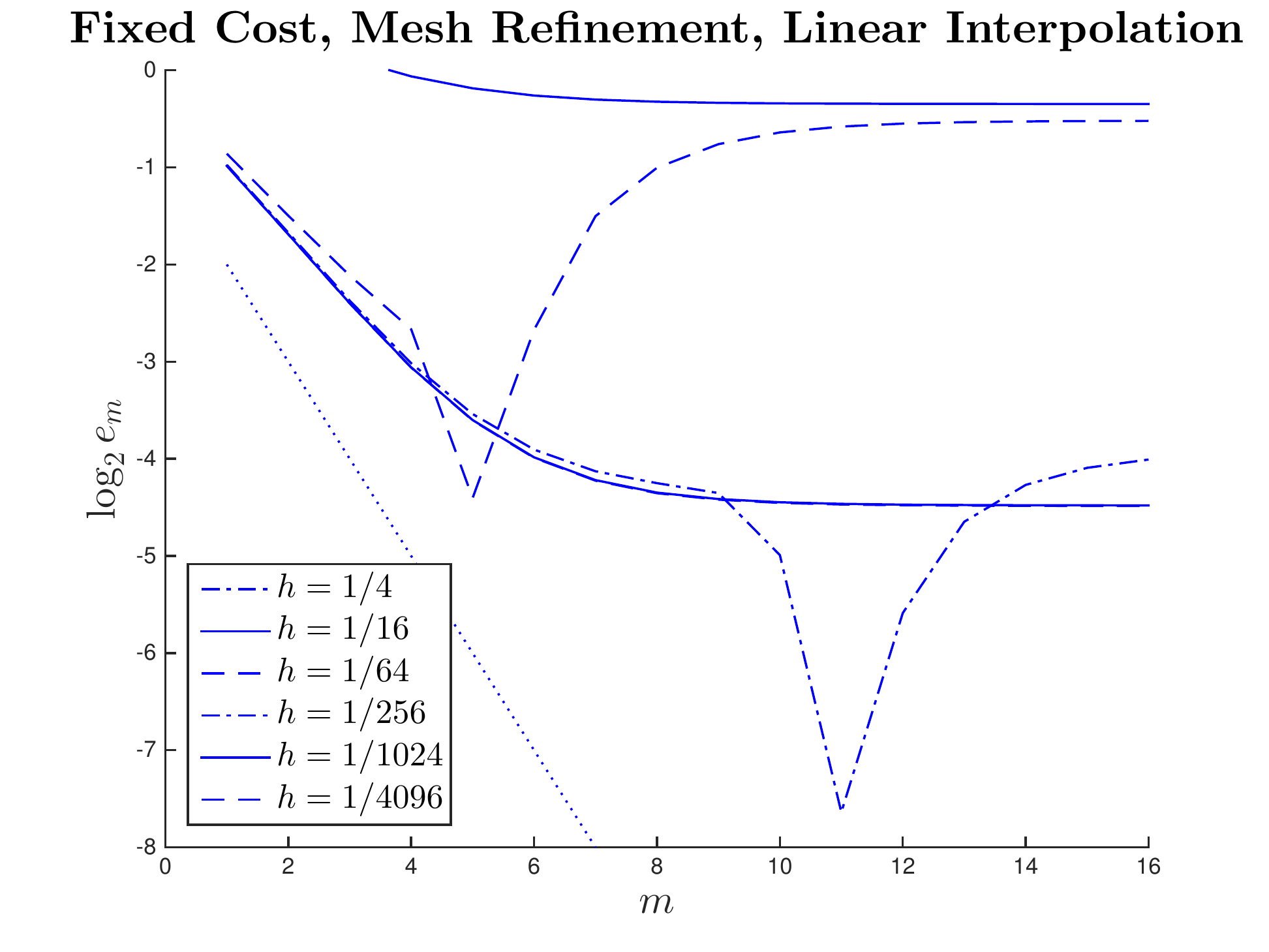}  \hfill
\includegraphics[width=0.495\columnwidth]{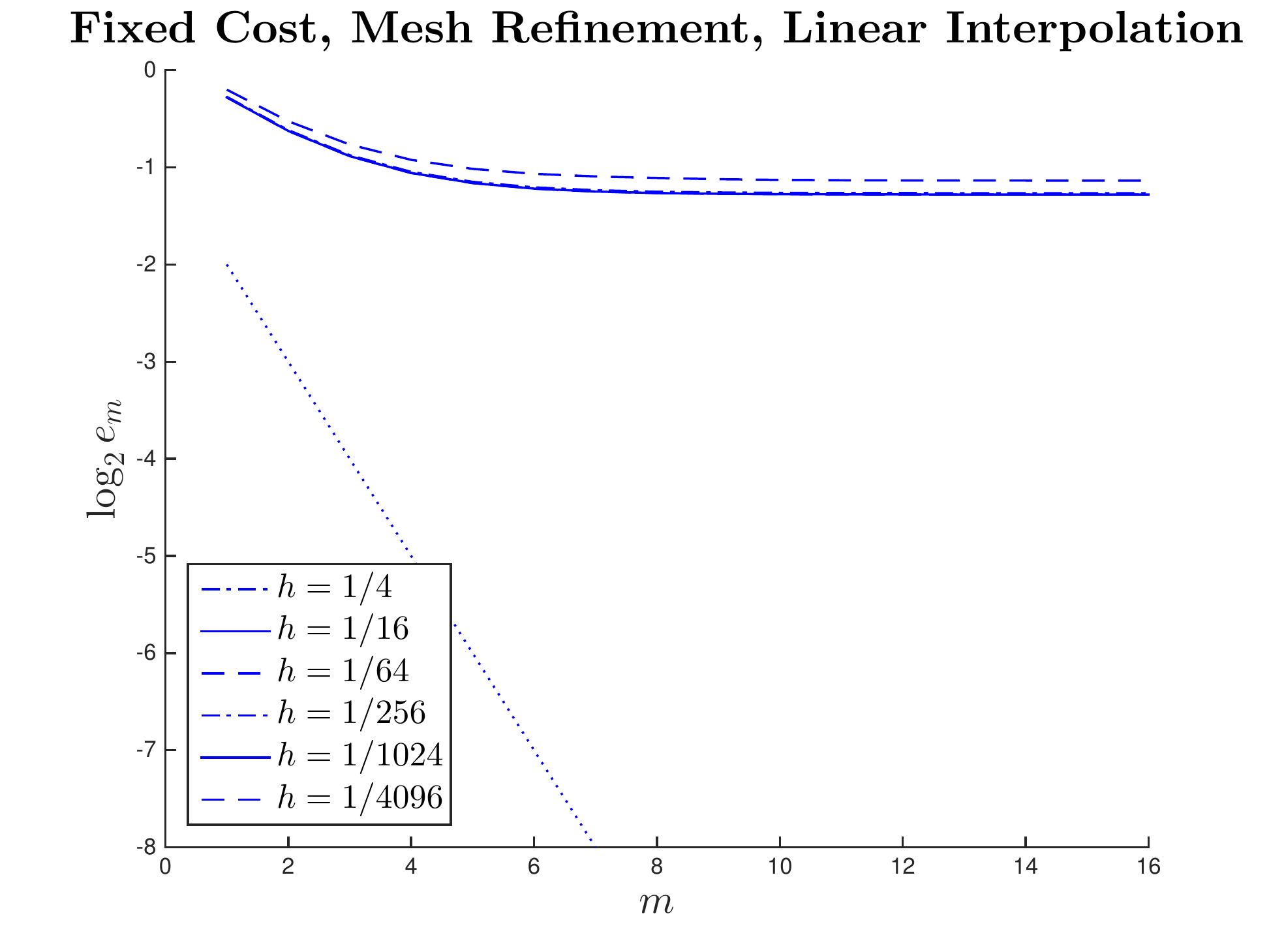} \\
\caption{
The uncertain volatility test case with parameters as in Table \ref{tab:paramuv}.
Piecewise constant policy timestepping with linear interpolation between individual meshes for each control.
Shown is the $\log_2$ error,
where $e_m = e((\Delta \tau)_m,h,c)$ from (\ref{uverror}) is the error
for timestep $(\Delta \tau)_m=1/8 \cdot 2^{-m} \in \{1/8,\ldots, 1/524288\}$ and
in each plot, from top to bottom, $h=1/4,1/16,1/64,1/256,1/1024,1/4096$. The cost size is fixed at $c=0.01$ (left) 
and $c=0.16$ (right).
The dotted line has slope $-1$.
The downward spikes are a result of error cancellation for a particular combination of $h$, $\Delta \tau$ and $c$.
}
\label{fig:meshcostplots}
\end{figure}
We can make a number of observations.
For large switching cost (right plot), the difference between the switching system and the HJB equation is large
and dominates the discretization error.
For fixed $c$ and $h$, there appears to be convergence as $\Delta \tau \rightarrow 0$, and if we also let $h\rightarrow 0$
the solutions converge to the solution of the switching system with fixed $c$.
For comparable values for $h$ and $\Delta \tau$, there appears to be a cancellation of leading order errors in $h$ and $\Delta \tau$ with opposite signs, which appears as downward spikes in the left and middle plot.

In this particular case of linear interpolation, 
since the overall scheme is monotone, convergence to the
viscosity solution of (\ref{uncertainpde1d}) is ensured
if $c=0$
as long as  the discretization is consistent.
Recall from Remark \ref{c_0_error}, that (for smooth
solutions) the discretization  error is of the form 
$O(\Delta \tau)$ + $O(h^2)$ + $O(h^2/\Delta \tau)$,
with the third term being the cumulative effect of linear
interpolation over $O(1/\Delta \tau)$ timesteps. 
Hence we can ensure consistency by requiring
that  $\Delta \tau = O(h)$.

We now analyze in more detail the convergence in $\Delta \tau$ and $h$ for 
the degenerate case with $c=0$.
The computational results are shown in Fig.~\ref{fig:1dplots} and are discussed in the following list.
\begin{enumerate}
\item
{\bf Policy timestepping, fixed mesh}:  We use (\ref{uvpw}) on a single uniform mesh on $[\log(K) - 4\cdot \bar{\sigma},
\log(K) + 4\cdot \bar{\sigma}]$, i.e., encompassing four standard deviations either side, where $\bar{\sigma}$ is the average of the two extreme volatilities. 
For fixed mesh size, the error approaches a constant level for decreasing time-step, but for simultaneously diminishing mesh size the observed time discretization error is clearly of first order in the timestep.

\item
{\bf Direct control}: Here, the optimal control is implicitly found with the solution as described above -- see, in particular, (\ref{directcontr}) -- and therefore we can disentangle the Euler discretization error from the effect of piecewise constant control.
Comparing the envelope to the curves, parallel to the dotted line with slope minus one, shows first order convergence as in the previous case,
but with a lower intercept which indicates that the time discretization error is about a factor 4 smaller.
Although the number of policy iterations per time-step was consistently small (usually 2--4), the computational time here was dramatically larger (due to the need to generate new matrices in each iteration)
and therefore solutions could not be computed for the same number of timesteps as for the other cases.

\item
{\bf Policy timestepping, linear interpolation}:
\label{pollinint}
Again, first order convergence in the timestep is observed, however, 
the leading error terms are, from Remark \ref{c_0_error}, of the form $ O(h) +  O(\Delta \tau) + O(h^2/\Delta \tau)$.
As a result,  
convergence is ensured only if $h$ goes to zero faster 
than $\sqrt{\tau}$, and for $h\sim \Delta \tau$ first order convergence is expected.
Because of the maximum norm stability and linear interpolation, 
the error does not explode even as $ \Delta \tau \rightarrow 0$ for fixed $h$. 
In fact, the solution goes to zero here as the interpolation 
introduces increasing artificial diffusion (see also Remark \ref{rem:swipen}) 
and the solution is absorbed at the boundaries.

\item
{\bf Policy timestepping, linear interpolation, reference mesh}:
Although not an issue for $J=2$ control parameters, if the dimension of the switching system is $J$, the number of interpolations from each mesh onto all other meshes is an $O(J^2)$ operation. We can avoid this using a single `reference mesh' to keep track of the solution.
So in addition to the two meshes associated with ${V}_{1}$ and ${V}_{2}$ as above under item \ref{pollinint}., we introduce a reference mesh, uniform on $[\log(K) - 4\cdot \bar{\sigma}, \log(K) + 4\cdot \bar{\sigma}]$, and $\widetilde{V}_{k,i}$ is constructed by linear interpolation from $V_k$ onto the reference mesh and then onto the $i$-th mesh point of the $j$-th mesh. With now two interpolations for each solution every timestep, convergence is still of first order but with a significantly higher factor than with direct interpolation between meshes. The number of interpolations needed for a $J$-dimensional switching system is now $O(J)$. 

\item
{\bf Policy timestepping, cubic interpolation}:
Finally, to reduce the accumulated interpolation error, we use the possibility of limited higher order interpolation for the mesh transfer afforded to us by Condition \ref{assumption:interpolation}, first without the use of a reference mesh. 
In particular, we use monotone piecewise cubic Hermite interpolation as in \cite{fritsch}. 
Note that the interpolation \cite{fritsch} is monotonicity preserving
but not  monotone 
in the viscosity sense \cite{barlessouganidis}.  
However,  Condition \ref{assumption:interpolation} is satisfied. 
Note that we use $c=0$ in this test case, which, strictly speaking,
does not ensure convergence to the viscosity solution.  However,
it is clear from Figure \ref{fig:costplots} that the limiting case of
$c \rightarrow 0$ does in fact converge to the viscosity solution,
hence it is interesting to include the case $c \equiv 0$.

The approximation order of the interpolation method
in \cite{fritsch} is guaranteed to be cubic only 
if the data are in fact monotone, and this is not the case for our initial data.
Nonetheless, the error is significantly reduced compared to the linear interpolation case.
\item
{\bf Policy timestepping, cubic interpolation, reference mesh}:
The results with cubic interpolation onto a reference mesh are not as accurate as for direct cubic interpolation between the computational meshes,
and have a similar accuracy to the results for linear interpolation without reference mesh.
\end{enumerate}

\begin{figure}
\includegraphics[width=0.495\columnwidth]{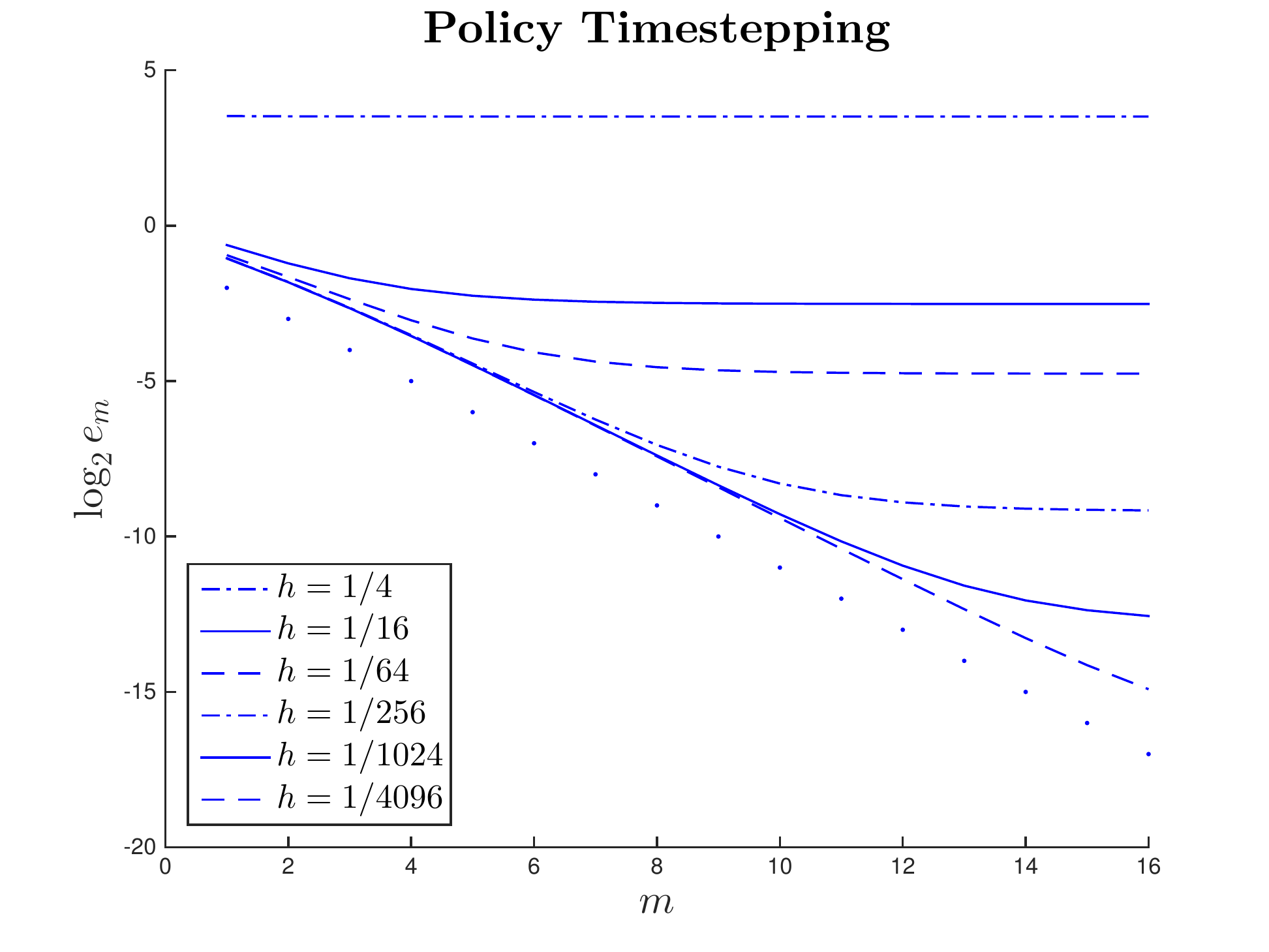}  \hfill
\includegraphics[width=0.495\columnwidth]{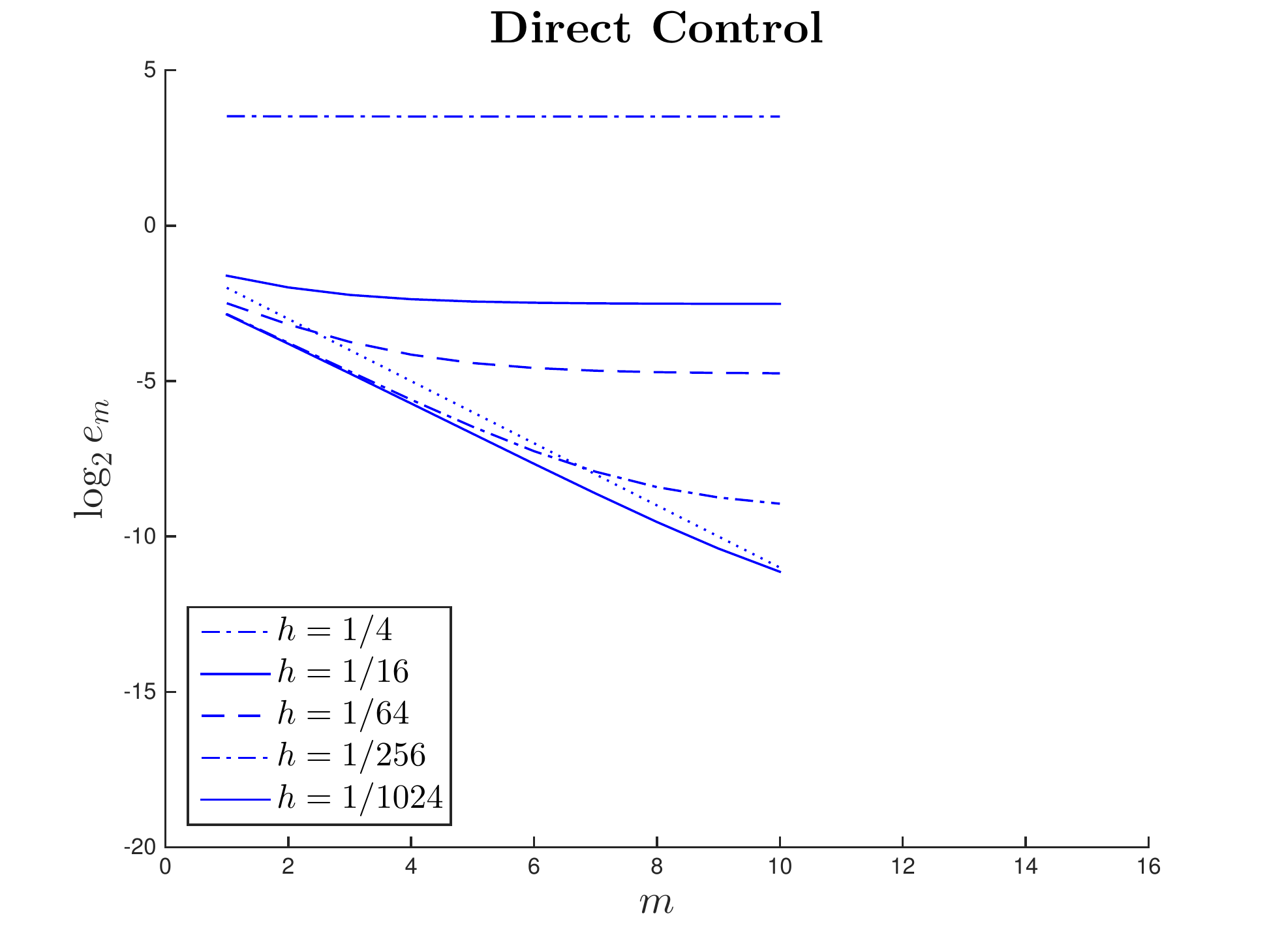} \\
\includegraphics[width=0.495\columnwidth]{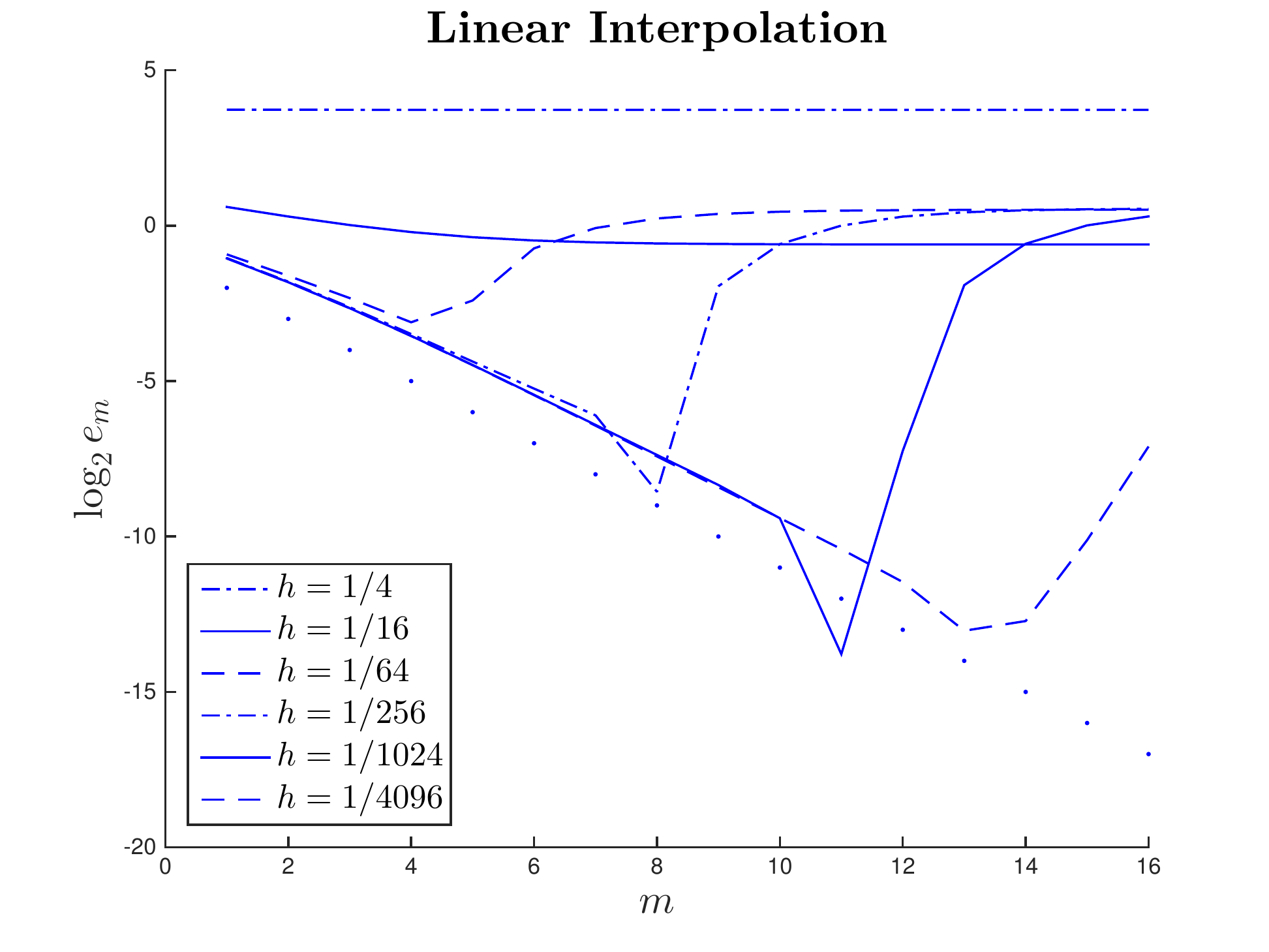} \hfill
\includegraphics[width=0.495\columnwidth]{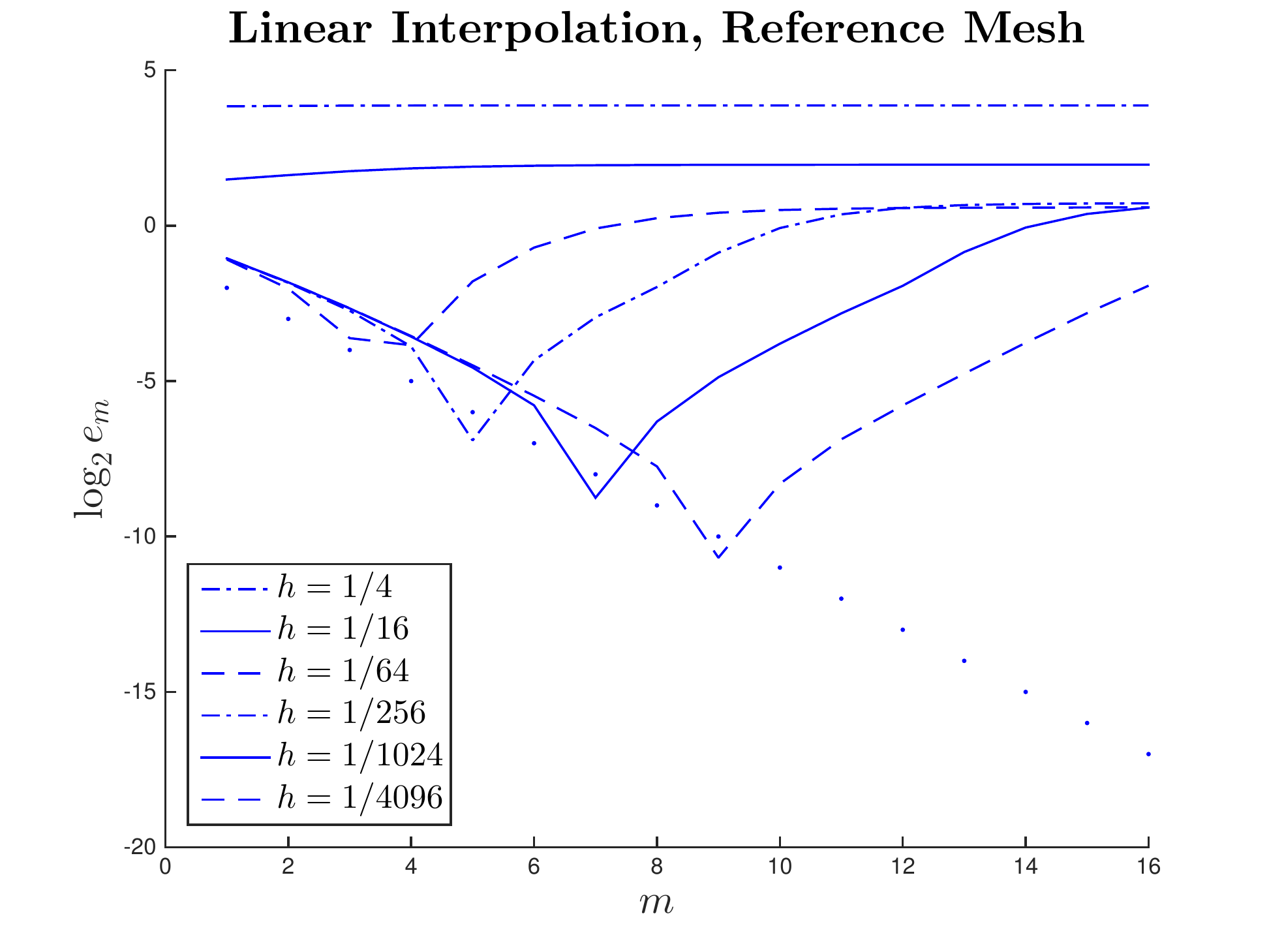} \\
\includegraphics[width=0.495\columnwidth]{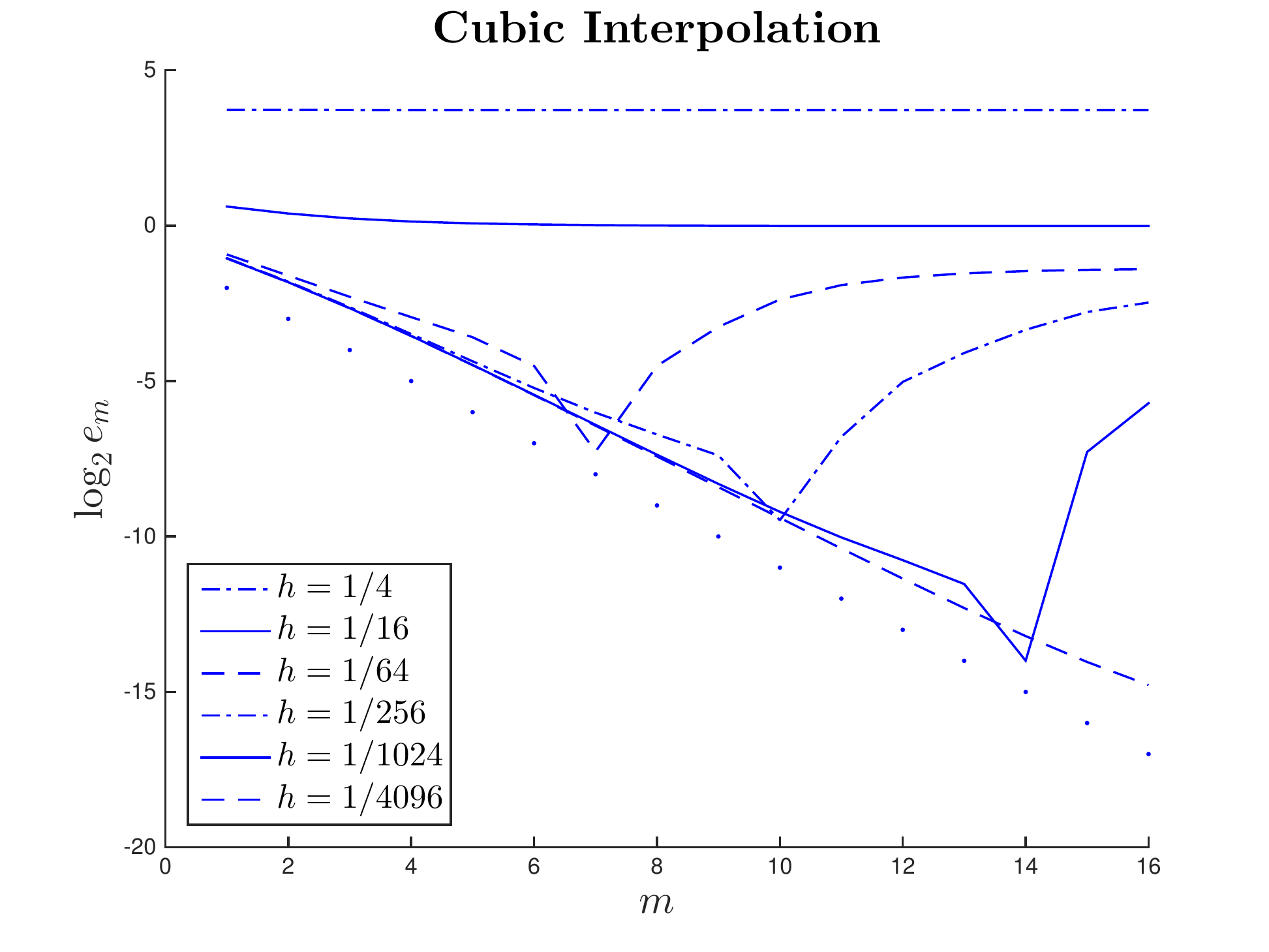} \hfill
\includegraphics[width=0.495\columnwidth]{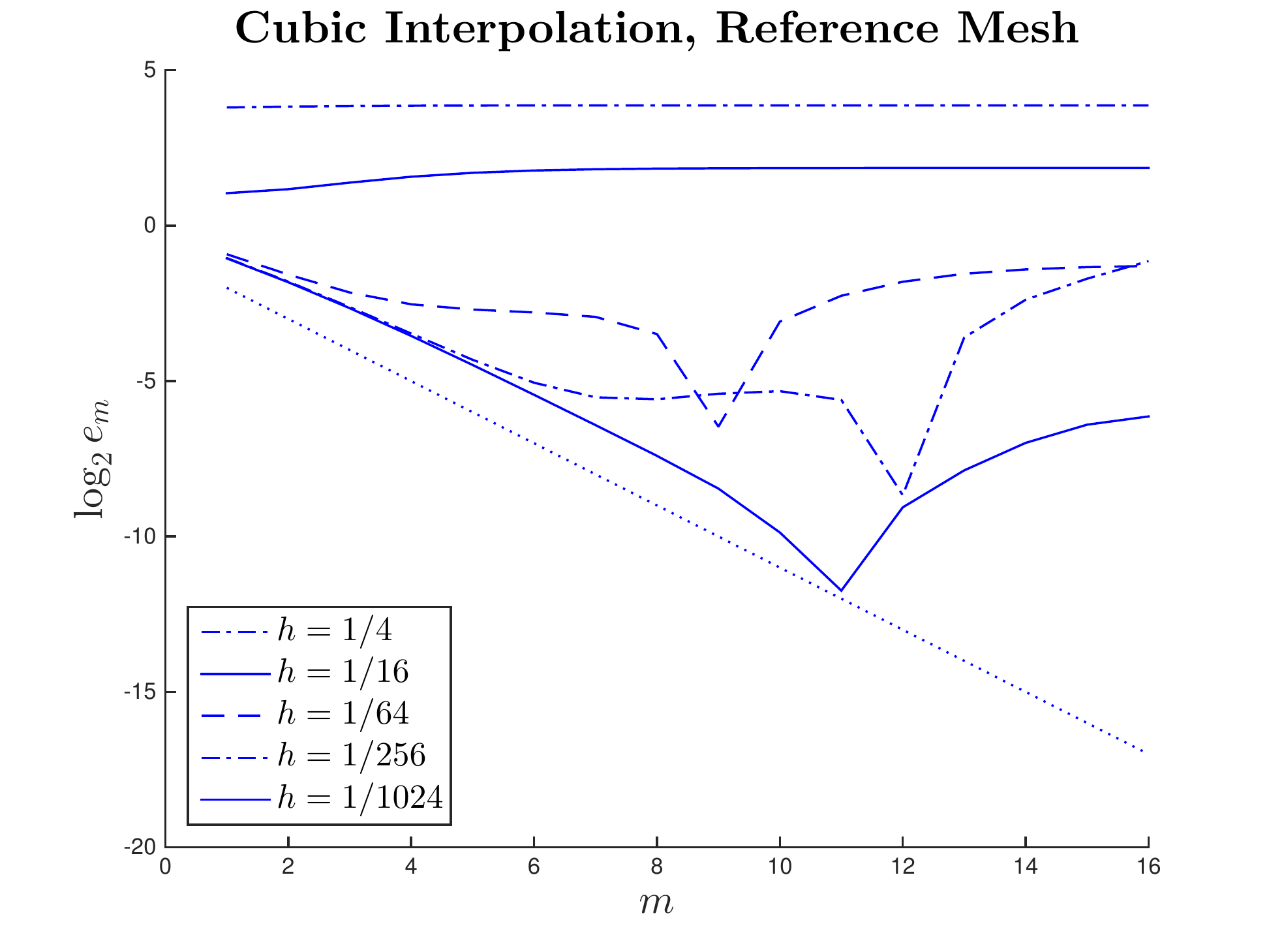}
\caption{The uncertain volatility test case with parameters as in Table \ref{tab:paramuv}. Shown is for different methods the $\log_2$ error,
where $e_m = e((\Delta \tau)_m,h)$ from (\ref{uverror}) is the error
for timestep $(\Delta \tau)_m=1/8 \cdot 2^{-m} \in \{1/8,\ldots, 1/524288\}$ and in each plot, from top to bottom, $h=1/4,1/16,1/64,1/256,1/1024,1/4096$.
The dotted line has slope $-1$.
The plots refer to, from top left lexicographically: the piecewise constant time-stepping method on a fixed mesh; the direct control method;
the piecewise constant time-stepping method with: linear interpolation; linear interpolation onto a reference mesh; cubic interpolation;
cubic interpolation onto a reference mesh.}
\label{fig:1dplots}
\end{figure}

\subsection{Mean-variance asset allocation}

As a second example we study the mean-variance asset allocation problem as discussed in
\cite{forsythwang10}, following the embedding technique introduced in \cite{ling00,zhouli00}.
In this example, we use the same grids for each constant policy mesh,
and focus on the effects of discretization of the control.
This example demonstrates that  piecewise constant policy 
timestepping does not introduce any significant extra error 
compared to first order Euler timestepping with either known 
optimal control, or a numerical optimal control obtained implicitly from the finite difference scheme.
We will also see that even a fairly coarse discretization of the
control admissible set yields good results. We have seen this property
of discretized controls in many examples.

The method determines the pre-commitment mean variance
optimal strategy \cite{Basak08}.
Note that it is possible to develop a numerical
method for solution of the
time-consistent version of this
problem \cite{wang2011}. However, since the time consistent
problem can be viewed as a constrained solution of the  pre-commitment
problem, the time consistent solution is sub-optimal compared to the
pre-commitment solution.
Specifically, we consider here the sub-problem given by the equation
\begin{eqnarray}
\label{meanvar1d}
\fpd{V}{\tau} - \inf_{p \in P} 
\mathcal{L}^{p} V &=& 0, \\
V(W,0) &=& \left(W-\frac{\gamma}{2}\right)^2,
\label{meanvarbc}
\end{eqnarray}
on $(-\infty,\infty)$, and $(0,\infty)$, with 
\begin{equation}
\label{meanvarL}
{L}^{p} V =
\harf \sigma^2 p^2 W^2 \spd{V}{W} 
+ (\pi + W (r + p \sigma \xi)) \fpd{V}{W}
\end{equation}
and either $p \in (-\infty,\infty)$ or $p \in [0,p_{\max}]$. 
Observe that (\ref{meanvar1d}) does not satisfy Assumption \ref{coeff_assump}
since $p$ can be unbounded.  We will re-parameterize the  control variable to
avoid this problem.

By solving equation (\ref{meanvar1d}--\ref{meanvarbc}) for various values
of the parameter $\gamma$, we can trace out the efficient frontier
in the expected value, variance plane \cite{forsythwang10}.

We use the standard finite difference discretization and make ``maximum'' use of central differences \cite{forsythwang08} whenever a positive coefficient scheme is achieved and use upwind differences only where necessary for monotonicity.


The PDE (\ref{meanvar1d}--\ref{meanvarbc}) is specified on
an infinite domain.  For numerical purposes, we approximate this
by means of a localized problem, with approximate boundary
conditions at finite values of $|W|$.
We use an asymptotic approximation of $V$
for large $|W|$.
In the cases we consider, an asymptotic value for the optimal control is more directly available than for the value function. Therefore, the following  solution under constant control will be applied as approximate boundary condition.
More precisely, the solution to the PDE
\begin{eqnarray}
\label{linpde}
\fpd{V}{\tau} - \harf a^2 W^2 \spd{V}{W} 
+ (\pi + b W) \fpd{V}{W} = 0
\end{eqnarray}
with terminal condition (\ref{meanvarbc}) is given by
\begin{eqnarray}
\label{linpdesol}
V(W,\tau) &=& \alpha(\tau)  W^2 +  \beta(\tau)  W +  \delta(\tau),
\end{eqnarray}
where $\tau=T-t$ and
\begin{eqnarray*}
\alpha(\tau)  &=& \exp((a^2+2b)\tau), \\
\beta(\tau)  &=& -(\gamma+c) \exp(b\tau) + c \exp((a^2+2b)\tau), \\
\delta(\tau)  &=& -\frac{\pi (\gamma+c)}{b} \left(\exp(b\tau)-1\right) + \frac{\pi c}{a^2+2b} \left(\exp((a^2+2b)\tau)-1\right) + \frac{\gamma^2}{4},
\text{ where}\\
c &=& 2 \pi/(a^2+b).
\end{eqnarray*}
In comparison to \cite{forsythwang10}, who only derive the highest-order term, this gives an asymptotically more accurate approximation and allows us to use substantially smaller domains for the computation.
Following \cite{forsythwang10}, we use the parameters in Table \ref{tab:parammeanvar} throughout.
\begin{table}
\centering
\begin{tabular}{|c|c|c|c|c|c|c|c|}
\hline 
$r$ & $\sigma$ & $\xi$ & $\pi$ & $W_0$ & $T$ & $\gamma$ & $\lambda$\\
\hline
$0.03$ & $0.15$ & $0.33$ & $0.1$ & $1$ & $20$ & $14.47$ & $1.762$ \\
\hline
\end{tabular}
\caption{Model parameters used in numerical experiments for mean-variance problem.}
\label{tab:parammeanvar}
\end{table}

%

We study two different cases for the permissible sets for state-variable and controls, one where
$W,p \in \mathbb{R}$, and one where $W\ge 0$, $p\in [0,p_{max}]$.

\subsubsection*{Bankruptcy allowed, unbounded control}

If bankruptcy ($W<0$) is allowed, the PDE (\ref{meanvar1d}--\ref{meanvarbc}) holds on $(-\infty,\infty)$.
In this case, a closed-form solution is known from \cite{hojgaardvigna}, where the optimal policy is given by
\begin{eqnarray}
\label{optpolanalyt}
p^\star(W,t) = - \frac{\xi}{\sigma W}\biggl[ W - \biggl( \frac{\gamma e^{ -r(T-t)} }{2} - \frac{\pi}{r} (1 - e^{-r(T-t)})
                                                 \biggl)
                                 \biggr].
\end{eqnarray}
Moreover, under this optimal policy, we find from the formulae in \cite{hojgaardvigna},
\begin{eqnarray*}
Var[W_T] &=& \frac{ e^{-\xi^2T}}{ 1 - e^{-\xi^2T} } 
            \biggl[
               E[W_T] - \biggl(
                             W_0 e^{rT} + \frac{\pi (e^{rT} - 1 ) }{r}
                        \biggr)
            \biggr]^2, \nonumber \\
E[W_T] &=& \biggl( W_0 + \frac{\pi}{r} \biggr) e^{ - (\xi^2-r)T} + \frac{\gamma ( 1 - e^{-\xi^2T})}{2} -\frac{\pi}{r} e^{-\xi^2T},
\end{eqnarray*}
such that $(\sqrt{Var[W_T]}, E[W_T]) = (0.794, 6.784)$, and
$E[(W_T-\gamma/2)^2] = Var[W_T] + E[W_T]^2 - \gamma E[W_T] + \gamma^2/4 = 0.8338$
for the parameters in Table \ref{tab:parammeanvar}.

As the optimal policy in the form (\ref{optpolanalyt}) is unbounded, we perform the control discretization in a different control variable. Noting that $p^* W$ is bounded as $W\rightarrow 0$, it seems natural to consider $p W$ as control variable in this area; however, $p^* W \sim -\xi/\sigma W$ as $|W|\rightarrow \infty$. This leads us to consider
\begin{eqnarray}
\label{contrtrans}
q = \frac{p W}{\max(1,\omega |W|)}
\end{eqnarray}
as control variable for some $\omega>0$, and
\begin{equation}
\label{meanvarLq}
\widetilde{\mathcal{L}}^{q} V =
\harf \sigma^2 q^2 \max(1,\omega^2 W^2) \spd{V}{W} 
+ (\pi + W r + q \max(1,\omega |W|) \sigma \xi) \fpd{V}{W}.
\end{equation}

The optimal control $q^*(W,t)$ will be bounded 
on a localized domain,
and we fix an interval $Q=[q_{min},q_{max}]$ in which we search for the optimal control by a crude approximation.
In this whole process, a precise knowledge of the exact optimal control is not necessary, as we only use the rough asymptotic shape. 
For the computations below, we pick $\omega=5$ and $Q=[-2.5,3.5]$.  Since we solve the
PDE on a localized domain ($|W|$ bounded), and the control is now
bounded as well, the localized version of  equations (\ref{meanvar1d}--\ref{meanvarbc})
now satisfies Assumption \ref{coeff_assump}.

We note that this is an example where the optimal control is an unbounded function of the state variable, but by a suitable reformulation the piecewise constant policy timestepping method can still be applied, with the policy chosen from a bounded set. 

From (\ref{optpolanalyt}) one sees that $q^*(W,t) \rightarrow -\xi/\sigma$ for $|W|\rightarrow \infty$.
Therefore, asymptotically, (\ref{meanvar1d}), (\ref{meanvarL}) takes the form (\ref{linpde}) with suitable $a$ and $b$, obtained by inserting the constant asymptotic optimal policy.
We can then use the asymptotically exact boundary conditions (\ref{linpdesol}) 
for both $W_{max}$ and $W_{min}$.
We choose $W_{max}=40$ and $W_{min}=-40$ in the computations.

The discretized switching system has the form
\begin{eqnarray}
\label{meanvarpw}
\frac{ V_j^{n+1} - \min_{1\le k\le J}(V_k^{n} - c_{k,j} ) }{\Delta \tau} - L^h_{q_j} V^{n+1}_j = 0, \quad j=1,\ldots, J,
\end{eqnarray}
where $c_{k,j}$ is defined as in equation (\ref{uvpw}).
In this case, we can set the switching parameter $c_{k,j}=0$  
since no interpolation is used, and this reduces to conventional
piecewise constant policy timestepping \cite{krylov99}.
Then the numerical approximations to all $J$ components of the switching system are the same in each timestep after the minimum is taken.

Fig.~\ref{fig:valandpolicyplots} shows the value function $V$ and its asymptotic approximation for large $|W|$.
The two functions have visually identical tangents at the boundaries, and indeed experimentation with the values of $W_{min}$ and $W_{max}$ shows that the results around $W=1$ are not significantly affected by this approximation.

Also shown in Fig.~\ref{fig:valandpolicyplots} is the approximate optimal policy obtained 
numerically from the policy timestepping discretization
with 20 policy steps, and the exact formula (\ref{optpolanalyt}), transformed into a bounded control as per (\ref{contrtrans}).

\begin{figure}
\includegraphics[width=0.495\columnwidth]{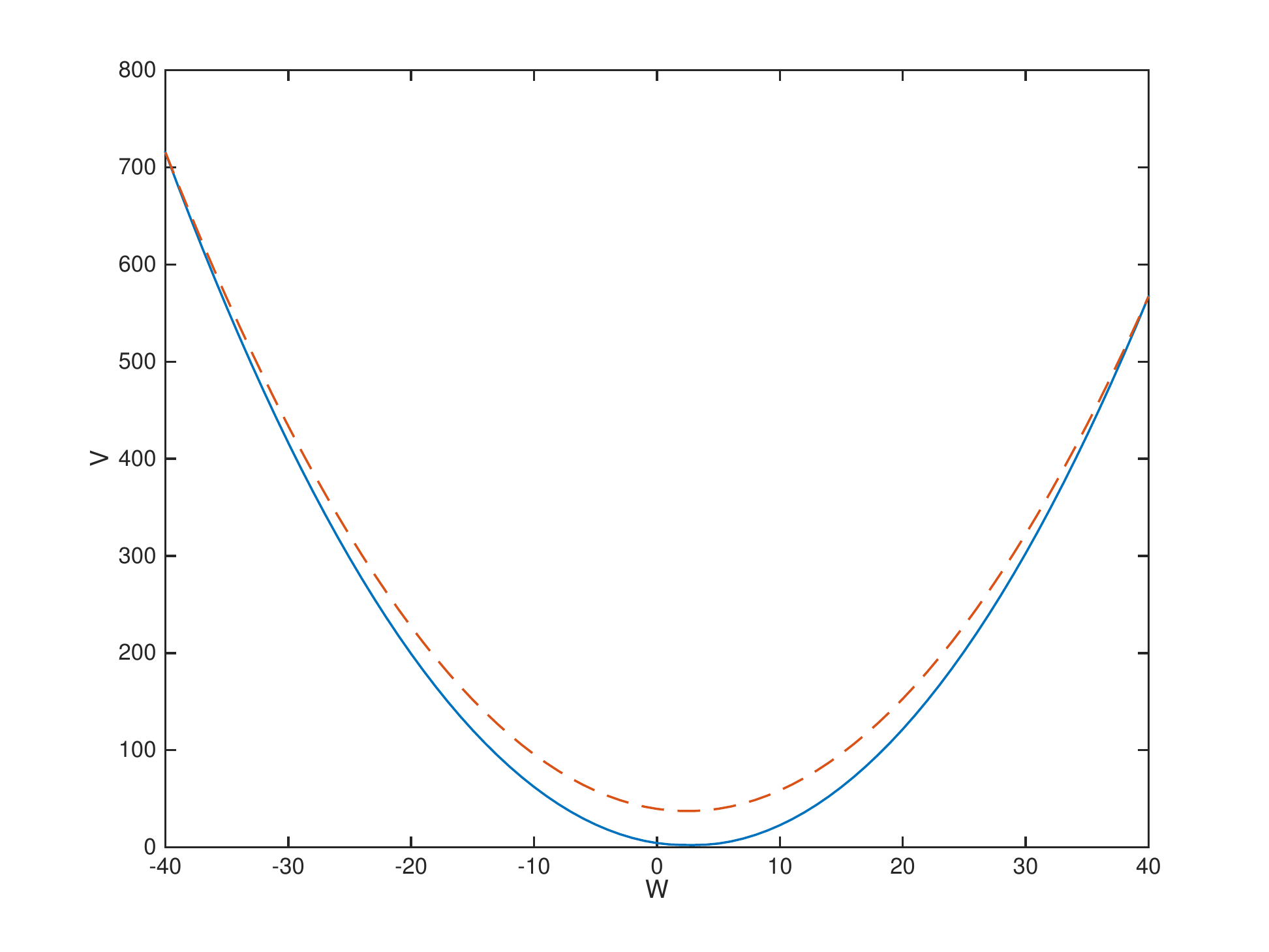}  \hfill
\includegraphics[width=0.495\columnwidth]{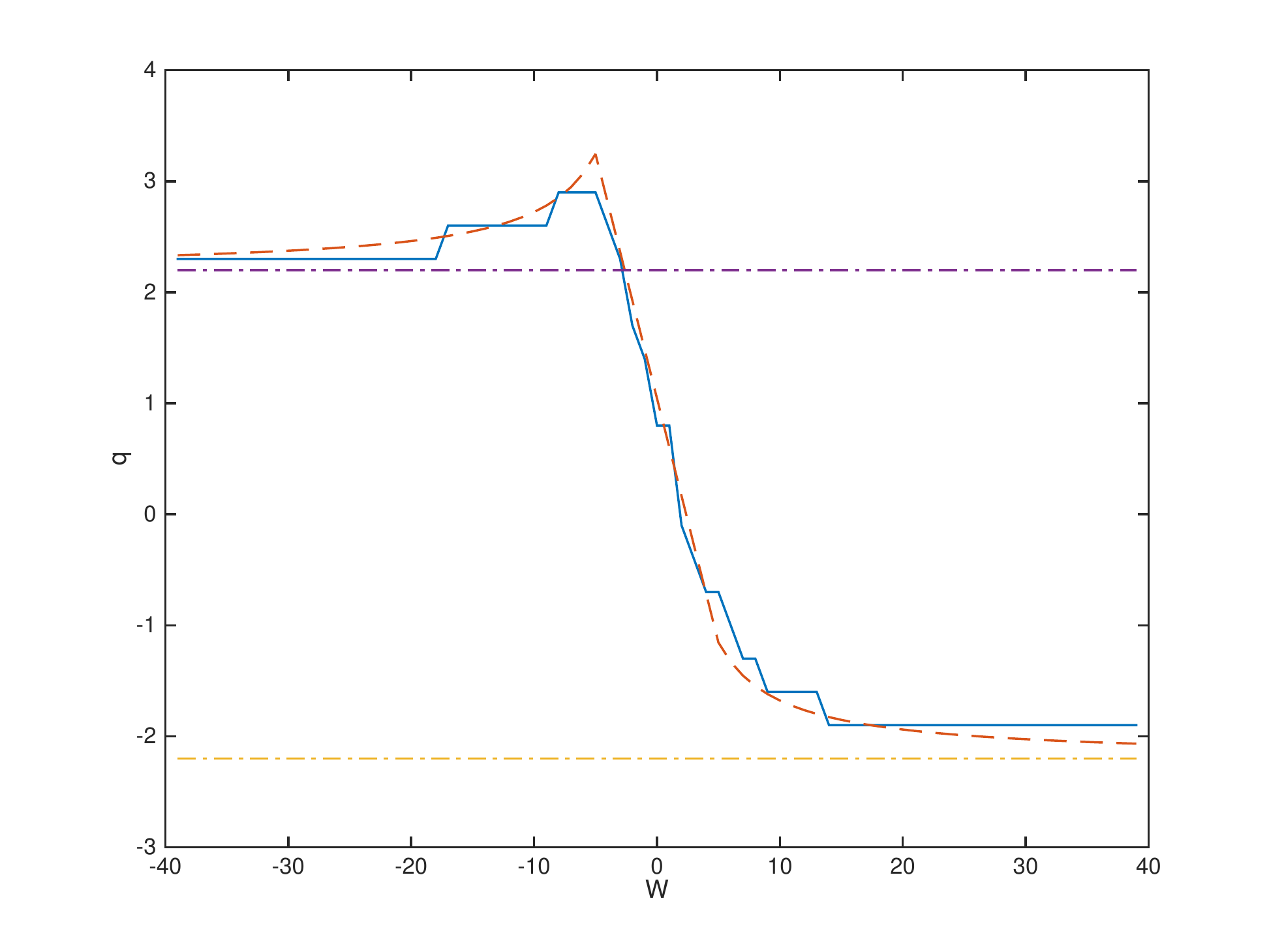}
\caption{
The mean-variance test case with parameters as in Table \ref{tab:parammeanvar}.
Left: The numerical approximation to $V(W,0)$ for piecewise constant policies for $M=80$, $N=80$ and 20 policy steps.
The dashed line is the asymptotic approximation for $|W|\rightarrow \infty$.
Right: The corresponding approximation to the optimal policy $q(W,0)$, and the analytical optimal policy (\ref{optpolanalyt}). The horizontal lines
are at the asymptotic optimal policies for $|W|\rightarrow \infty$.}
\label{fig:valandpolicyplots}
\end{figure}

Table \ref{tab:contrdiscrmeanvar} illustrates the convergence as the control mesh is refined for a fixed time and spatial mesh.
\begin{table}
\centering
\small
\begin{tabular}{|c|c|c|c|c|c|c|c|c|c|}
\hline
&$J=5$ & $J=8$ & $J=10$ & $J=15$ & $J=20$ & $J=29$ & $J=40$ & $J=57$ & $J=80$ \\
\hline \hline
(a)&2.257 & 1.531 & 1.429 &1.254 &1.230 &1.196 &1.186 &1.180 &1.178 \\
(b)&&-0.725 &   -0.101 &   -0.175 &   -0.0241 &  -0.0339 &  -0.0104  & -5.4 $\cdot 10^{-3}$ &  -2.5 $\cdot 10^{-3}$ \\
(c)&&& 7.12 &    0.58 &  7.26 &   0.71 &   3.25 &   1.92 &    2.16 \\ \hline
\end{tabular}
\caption{The mean-variance test case with parameters as in Table \ref{tab:parammeanvar}.
Convergence of the control discretization alone for $N=480$, $M=120$. Shown are:
(a) the numerical solution $V_k = \tilde{V}(N,M,J_k;W_0,0)$; 
(b) the increments $V_{k}-V_{k-1}$; 
(c) the ratios $(V_{k}-V_{k-1})/(V_{k-1}-V_{k-2})$,
 for $J_k = \Big\lceil  5 \cdot \sqrt{2}^{k-1} \Big\rceil$.}
\label{tab:contrdiscrmeanvar}
\end{table}
The estimated order of convergence over these refinement levels is 2.
We pick $J=40$ fixed for the following tests of the convergence 
in the mesh size and timestep. 
For this value, the control discretization error was empirically negligible (compared to the time and spatial discretization error).

Fig.~\ref{fig:MVconv} shows the convergence of the approximations for piecewise constant policy timestepping and for the use of the exact policy 
given by (\ref{optpolanalyt}). In the latter case, the error is solely due to the Euler time-discretization and spatial finite differences.
For piecewise constant timestepping, 40 policies were used, so that the computational time for the same mesh is about a factor of 40 larger
than for a single policy, hence we show slightly fewer refinements. 
It appears that the spatial approximation error for large mesh size $h$ is smaller if knowledge of the optimal control is used. The envelope showing the time discretization error is consistent with first order convergence. Interestingly, the intercept is about 4 units higher for the exact policy, so that the results using policy timestepping are about a factor 16 more accurate for the same timestep. This is not to be expected generally and must result from opposite signs of the Euler truncation error and the error due to piecewise constant policies.

\begin{figure}
\includegraphics[width=0.495\columnwidth]{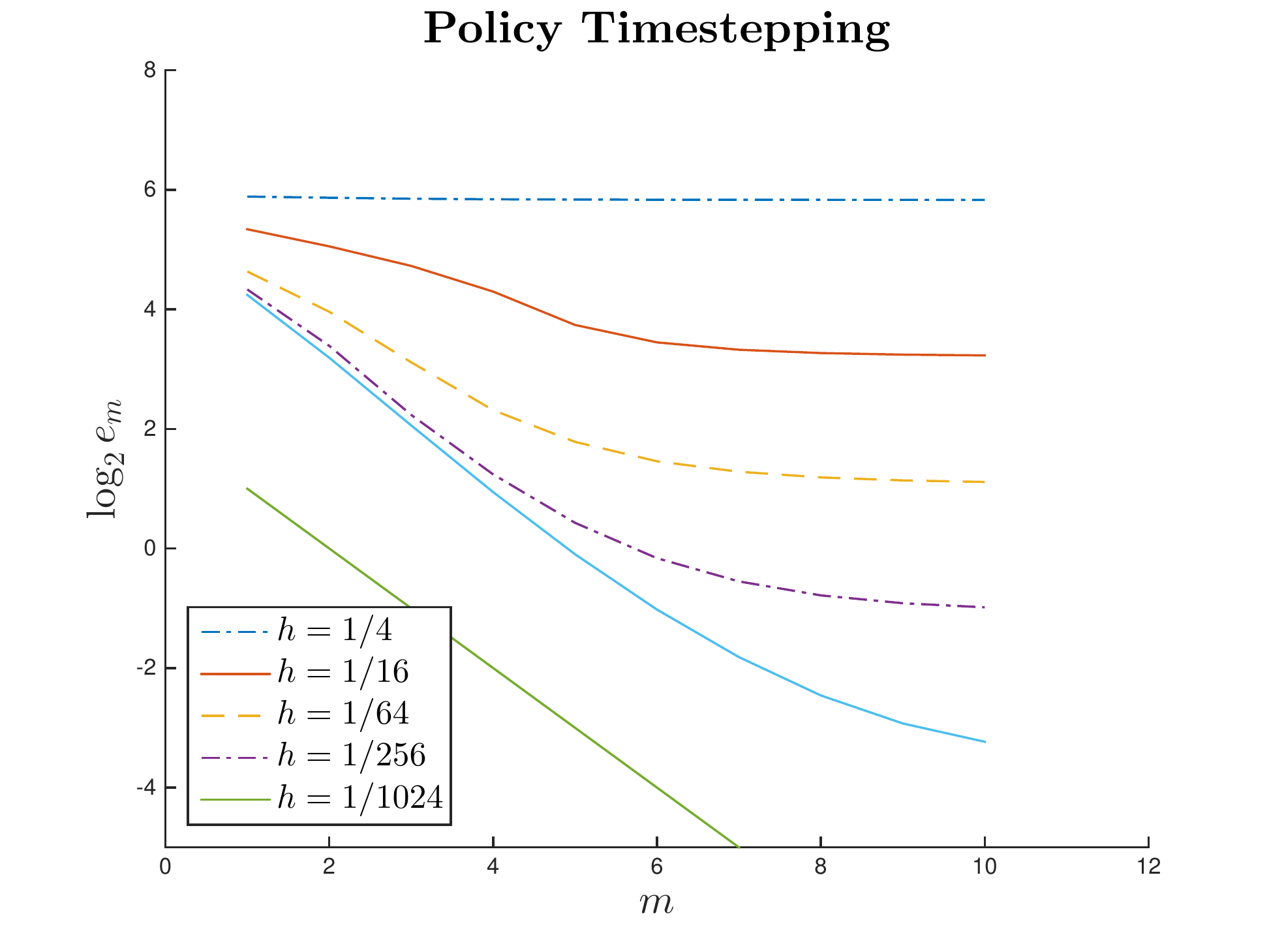}  \hfill
\includegraphics[width=0.495\columnwidth]{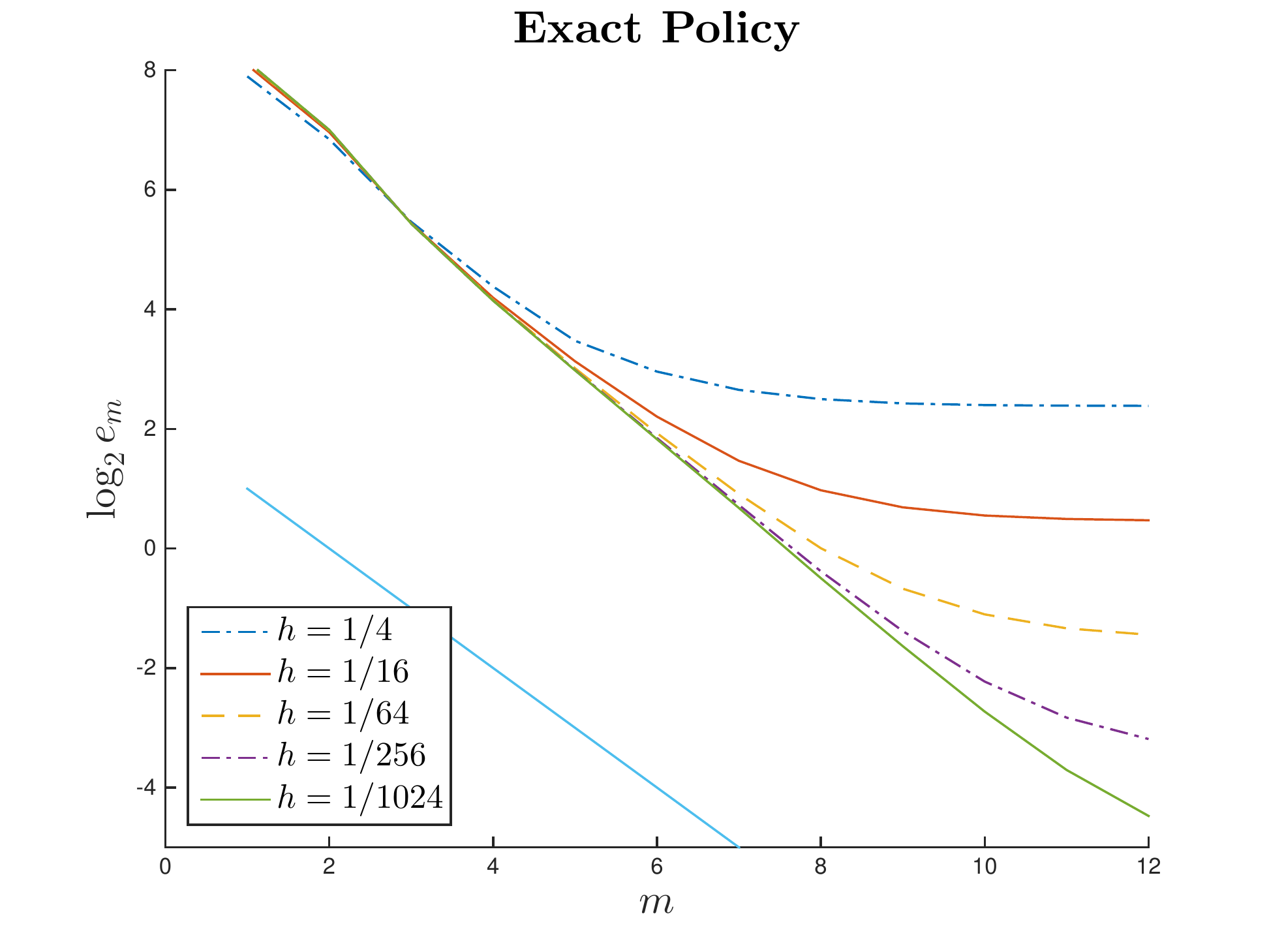}
\caption{
The mean-variance test case with parameters as in Table \ref{tab:parammeanvar}.
Similar to Fig.~\ref{fig:1dplots} in the previous section, the $\log_2$ error for $(\Delta \tau)_m=1/8 \cdot 2^{-m} \in \{1/8,\ldots, 1/32768\}$ and in each plot, from top to bottom, $h=1/4,1/16,1/64,1/256,1/1024$.
The straight line has slope $-1$.
Left: Piecewise constant policy timestepping with 40 equally spaced policies in $[-2.5, 3.5]$.
Right: Using the exact policy given by (\ref{optpolanalyt}).}
\label{fig:MVconv}
\end{figure}

%

\subsubsection*{No bankruptcy, bounded control}

If bankruptcy ($W<0$) is not allowed, the PDE (\ref{meanvar1d}--\ref{meanvarbc}) holds on $(0,\infty)$.
The boundary equation at $W=0$ is then
\begin{eqnarray}
\label{bdypde}
V_{\tau}(0,\tau) - \pi V_W(0,\tau) = 0,
\end{eqnarray}
see  \cite{forsythwang10} for a discussion.
For $\pi>0$ there is an outgoing characteristic (going backwards in time) so that no boundary condition is required and we can approximate (\ref{bdypde})
by upwind differences from interior mesh points. In fact, as we are switching to upwind differences locally whenever the monotonicity of the scheme is violated (see above and \cite{forsythwang08}), upwinding will always be used for small $W$ if $\pi>0$.

For bounded control with no short-selling, $P=[0,p_{max}]$ in (\ref{meanvar1d}). 
In the computations, we choose $p_{max}=1.5$ as an 
attained upper bound (the same used in \cite{forsythwang10}).  This would
correspond to a typical leverage constraint.
For large $W$, we use again the approximation (\ref{linpdesol}), with coefficients based on the asymptotic optimal control $p=0$
(see Fig.~\ref{fig:MVBddconv}).

The numerically computed value function (a closed-form solution is not available in this case) is shown in Fig.~\ref{fig:MVBddconv},
together with the asymptotic approximation for large $W$. Also shown is the numerically computed optimal control.

\begin{figure}
\includegraphics[width=0.495\columnwidth]{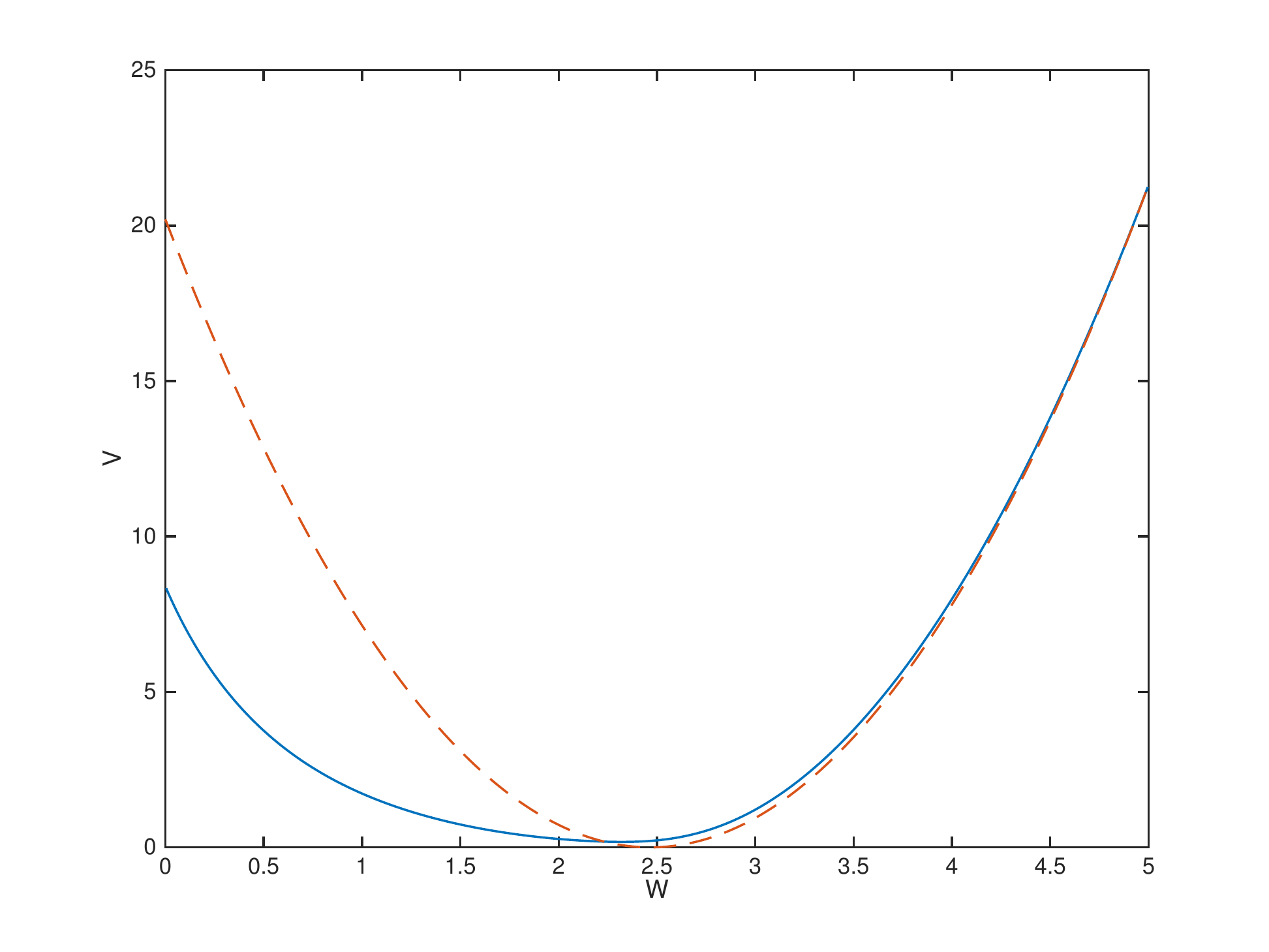}  \hfill
\includegraphics[width=0.495\columnwidth]{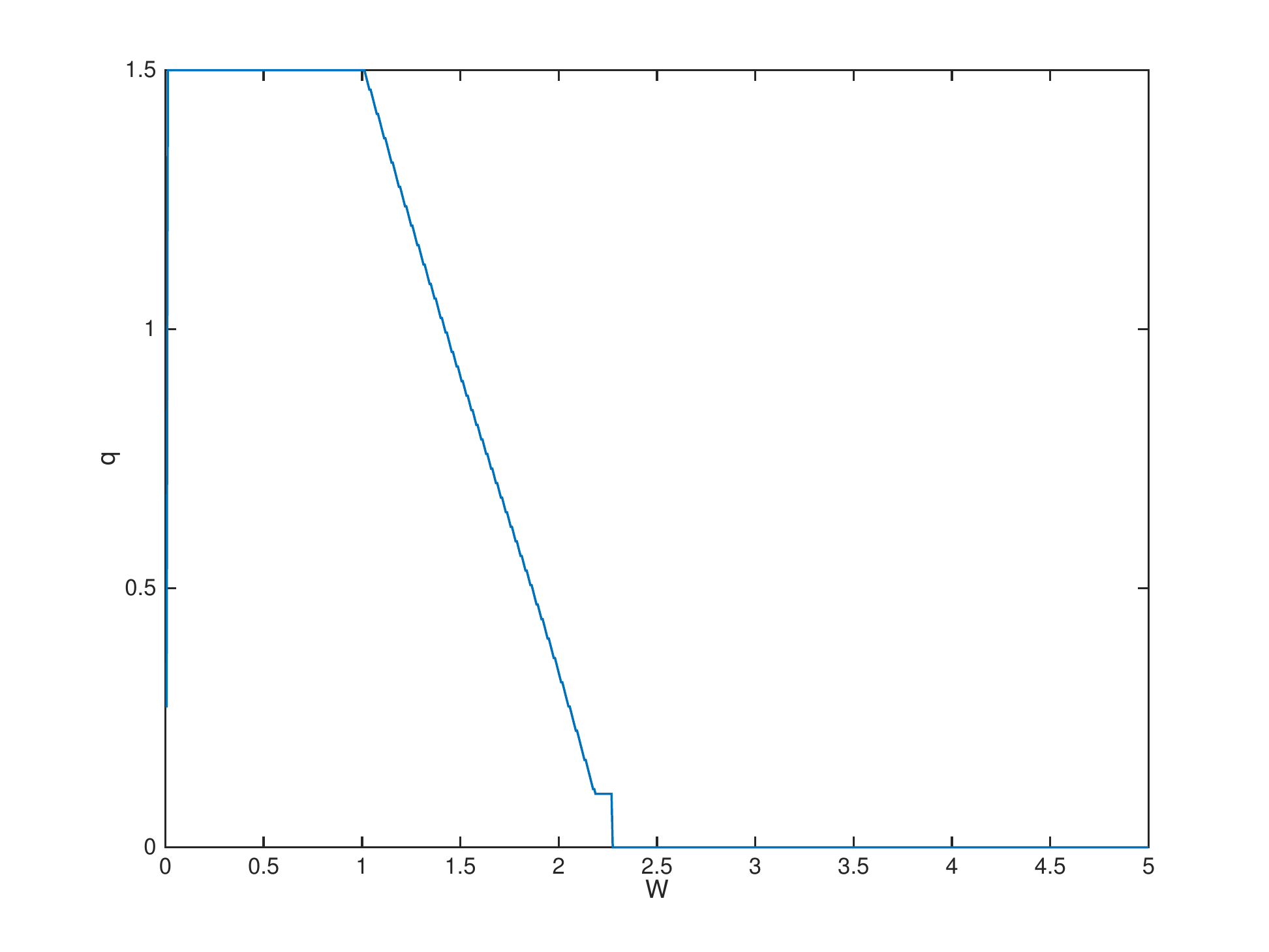}
\caption{
The mean-variance test case with parameters as in Table \ref{tab:parammeanvar},
with no bankruptcy and bounded control. 
 All model parameters are as in Table~\ref{tab:parammeanvar}.
Left: Value function $V(W,0)$ (solid line) and the asymptotic approximation (dashed line) for large $W$.
Right: The numerical optimal policy $p(W,0)$.}
\label{fig:MVBddconv}
\end{figure}

We compare the results achieved by piecewise policy timestepping to those achieved by the direct control formulation.
For clarity, the two discretizations used are
\begin{eqnarray}
\label{meanvardirect}
\frac{V^{n+1}-V^n}{\Delta \tau} - \min_{q \in Q_h}
L^h_{q} V^{n+1} &=& 0
\end{eqnarray}
for the direct control method and (\ref{meanvarpw}) for the piecewise constant timestepping method.

We use policy iteration as in \cite{bokanowski} to solve the discrete control problem in (\ref{meanvardirect}).
We use a positive coefficient discretization \cite{forsythwang08} with central differencing
used as much as possible.  For the direct control method,  and
the piecewise constant policy timestepping method,
it is straightforward to
verify that the discretization is monotone, consistent and stable \cite{forsythwang08,krylov99}.
The results are shown in Table \ref{tab:totaldiscrmeanvarbdd}.

\begin{table}
\centering
\begin{tabular}{|l|c|c|c|c|c|c|}
\hline
&& $M=800$ & $M=1600$ &$M=3200$ &$M=6400$ &$M=12800$ \\
&& $N = 50$ &  $N = 100$ & $N = 200$ & $N = 400$ & $N = 800$ \\
&& $J = 5$ & $J = 8$ & $J = 10$ & $J = 15$ & $J = 20$ \\
\hline
\hline
Policy & (a) & 1.5930 &   1.5589  &  1.5447  &  1.5378 &   1.5350 \\
timestepping & (b) &&   -0.0341  & -0.0141 &  -0.0069 &  -0.0028 \\
& (c) &&& 2.41   & 2.04  &  2.45 \\
\hline
Direct & (a) & 1.5902  &  1.5577  &  1.5442   & 1.5376 & $\star$ \\
control & (b) && -0.0326  & -0.0135 &  -0.0066 & $\star$ \\
& (c) &&& 2.4167  &  2.0333 & $\star$ \\
\hline
Fixed & (a) &3.4268 &   3.4199  &  3.4140  & 3.4104 & 3.4085 \\
control & (b) && -0.0069  & -0.0059 &  -0.0036 &  -0.0019 \\
($q=1.5$) & (c) &&&1.1733  &  1.6523  &  1.8399 \\
\hline
\end{tabular}
\caption{
The mean-variance test case with parameters as in Table \ref{tab:parammeanvar},
with no bankruptcy and bounded control. 
Shown are, for the policy timestepping method, the direct control method, and for a fixed constant control:
 (a) the numerical solution $V_k = \tilde{V}(N_k,M_k,J_k;W_0,0)$; (b) the increments $V_{k}-V_{k-1}$; 
 (c) the ratios $(V_{k}-V_{k-1})/(V_{k-1}-V_{k-2})$
 for 
 $M_k = 800 \cdot 2^{k-1}$,
 $N_k = 50 \cdot 2^{k-1}$,
 $J_k = \Big\lceil  5 \cdot \sqrt{2}^{k-1} \Big\rceil$
 (except for the fixed control case, where $J=1$).
}
\label{tab:totaldiscrmeanvarbdd}
\end{table}

In each step of the policy iteration, the maximum (over parameters) of the discretized differential operator at any given mesh-point has to be computed. As the discretization (local upwinding based on the coefficients) depends on the control parameter in a discontinuous way, this maximum is found by discretizing the control and  exhaustive search.
This makes the complexity of  a single policy iteration identical to a 
single timestep of the constant policy timestepping algorithm.
Thus, overall, the typically 
observed 4--6 iterations in every timestep 
translates into a 4--6 factor of 
increase in the CPU cost of the direct control method compared
to the piecewise constant policy timestepping technique. 
Due to this increased cost, we do not show the direct control results 
for the finest level (marked $\star$).

The refinements were chosen such that at the coarsest level a single separate refinement of the spatial mesh, timestep and control mesh gave comparable (empirical) accuracy improvements. This ensures that the data test the convergence order in all three discretization parameters.
It is clear that the achieved accuracy is almost identical for both methods.

We also include results for the value achieved with  a fixed control, $q=1.5$. 
This is the chosen upper bound and the optimal value attained in an interval 
around $W_0=1$, see Fig.~\ref{fig:MVBddconv}. The results are distinctly different from 
those under the optimal control, which shows that the similar performance 
of policy timestepping and direct control is not a result of the control being 
constant near $W=1$.
The errors for fixed control are purely due to the time and spatial finite 
difference discretization, and are slightly smaller than those observed in the true optimal control problems.

%
%
%
%
%
%
%
%
%
%
%
%
%
%
%
%
%
%

%
%
%
%
%
%
%
%



\section{Conclusions}
\label{sec:discuss}

This article analyzes the piecewise constant policy timestepping 
method both from a theoretical and an applications perspective.
Our main result is that if we use different meshes for each constant policy PDE solve,
then convergence to the viscosity solution can be
proven even if high order (not necessarily monotone) interpolation techniques are used. 
Essentially, this is because we can view the piecewise constant policy timestepping
method as the solution to a switching system of PDEs, where the coupling
between the PDEs occurs only in the zeroth order term.
However, this generality comes at a price: we must include
a finite switching cost in the switching system.  Convergence
to the solution of the original HJB PDE occurs only in the
limit as the switching cost tends to zero.  However,
our numerical experiments show  that good results
are obtained for very small (even zero) switching costs.

The general approach we follow
also has superficial similarities with the ``semi-Lagrangian methods'' (SLM) of \cite{camilli1995approximation} and \cite{debrabantjakobsen12}.
They both make use of the fact that for \emph{given} coefficients (controls), it may be easier to construct monotone schemes together with the underlying mesh, especially in more than one dimension.
If different controls require different meshes, interpolation of the mesh solution is needed in every timestep. In the present method this serves to carry out the optimization over solutions with different policies.

The computational results demonstrate that a smaller error
is obtained using the high order interpolation, compared
to linear interpolation. 

In many practical situations, the local optimization
problem at each node is determined by discretizing the
control and using exhaustive search.  In this case,
our tests  show that
piecewise constant policy timestepping is more efficient than 
standard direct control methods, as a similar level of accuracy
is achieved with less computational effort. This is simply due
to the fact that piecewise constant policy timestepping is
unconditionally stable, and does not require a policy
iteration to solve nonlinear discretized equations.

The use of piecewise constant policy timestepping
can be useful in situations where generic monotone schemes 
are hard to construct, e.g., in multidimensional settings, whose 
implementation we do not consider here and leave for future work.

Finally, we note that it is straightforward to implement piecewise constant
policy timestepping in existing linear PDE solution software.  Hence these
existing algorithms can be easily converted to solve nonlinear
HJB equations.

\appendix
\section{Proof of Lemma \ref{lipshitz_assump}}
\label{lemma1_proof}

We provide here a proof of Lemma \ref{lipshitz_assump},
\begin{proof}
By insertion one gets
\begin{eqnarray*}
   | F({\bf{x}}, \phi( {\bf{x}} ),  D \phi( {\bf{x}} ), D^2 \phi( {\bf{x}} )  ) 
      - F_h( {\bf{x}}, \phi({\bf{x}}) + \xi ,  D \phi( {\bf{x}} ), D^2 \phi( {\bf{x}} )  )   | 
      &=&
      \left|
      \sup_{q\in Q_h} L_q (\phi + \xi) - \sup_{q\in Q} L_q \phi
      \right|
      \\
& \hspace{-10 cm} \leq&   \hspace{-5 cm}  
            \left|
      \sup_{q\in Q_h} L_q (\phi + \xi) - \sup_{q\in Q} L_q (\phi + \xi)
      \right|
      +
            \left|
      \sup_{q\in Q} L_q (\phi + \xi) - \sup_{q\in Q} L_q \phi
      \right|
\end{eqnarray*}
by the triangle inequality.
From Assumption \ref{coeff_assump} and
the compactness of $Q$, then the supremum of $L_q \phi$ is attained, say at $q^*$, and then
\[
L_{q^*} \phi - r_{q^*} \xi = L_{q^*} (\phi + \xi) \le
 \sup_{q\in Q} L_q (\phi + \xi) \le  \sup_{q\in Q} L_q \phi + \sup_{q\in Q} L_q \xi
 = L_{q^*} \phi + \sup_{q\in Q} (-r_{q}) \xi,
\]
hence
\begin{eqnarray}
\label{firstpart}     
            \left|
      \sup_{q\in Q} L_q (\phi + \xi) - \sup_{q\in Q} L_q \phi
      \right|      
      &\le & \xi \sup_{q\in Q} |r_q|.
\end{eqnarray}

Now let $q_\xi^*$ be the maximizer of $L_q (\phi+\xi)$.
We also have by (\ref{denseset}) that there is $q_h^* \in Q_h$ with $|q_h^*-q_\xi^*|\le h$.
By uniform continuity of the coefficients in $q$ on the compact set $Q$, there exists a function
$\bar{\omega}_1$ so that
\[
\|\sigma_{q_\xi^*} \sigma_{q_\xi^*}^T
- \sigma_{q_h^*} \sigma_{q_h^*}^T
\| + \|\mu_{q_\xi^*}-\mu_{q_h^*}\| + |r_{q_\xi^*} -r_{q_h^*}| + |f_{q_\xi^*} - f_{q_h^*}| 
\le \bar{\omega}_1( {\bf{x}},h) \rightarrow 0, h\rightarrow 0
\]
with the usual vector and matrix norms, and hence
\[
\left|
L_{q_\xi^*} (\phi + \xi) - L_{q_h^*} (\phi + \xi)
\right|
\le
\bar{\omega}_1({\bf{x}},h)
\max\left(
1,
|\xi|,
|\phi|,
\|D \phi\|,
\|D^2 \phi\|
\right) \equiv \max(1,|\xi|) \bar{\omega}_2({\bf{x}},h)
\]
for a suitably defined $\bar{\omega}_2$.
Using also that $Q_h\subset Q$,
\[
\sup_{q\in Q} L_q (\phi + \xi) - \max(1,|\xi|) \bar{\omega}_2({\bf{x}}, h)
\le \sup_{q\in Q_h} L_q (\phi + \xi)
\le \sup_{q\in Q} L_q (\phi + \xi),
\]
so that
\begin{eqnarray}
\label{secpart}
\left|
      \sup_{q\in Q_h} L_q (\phi + \xi) - \sup_{q\in Q} L_q (\phi + \xi)
      \right| &\le&
      \max(1,|\xi|) \bar{\omega}_2({\bf{x}}, h)
      \; \le \; 
\bar{\omega}_2( {\bf{x}}, h) + \frac{\bar{\omega}_2^2({\bf{x}}, h)}{2} + \frac{\xi^2}{2}.
\end{eqnarray} 
From Assumption \ref{comparison_assump}, and  noting that the
      equation coefficients are uniformly
      continuous in $q$, it is easily shown
      that the right hand side
      of (\ref{secpart}) is locally Lipshitz in
      ${\bf{x}}$, independent of $h$.
The result then follows from (\ref{firstpart}) and (\ref{secpart}), with an appropriate choice of $\omega_h$ and $\omega_\xi$.

\end{proof}

\bibliography{uncvol,paper}

\begin{thebibliography}{10}

\bibitem{almgren2001}
R.~Almgren and N.~Chriss.
\newblock Optimal execution of portfolio transactions.
\newblock {\em Journal of Risk}, 3:5--40, 2001.

\bibitem{avellaneda95}
M.~Avellaneda, A.~Levy, and A.~Par\'{a}s.
\newblock Pricing and hedging derivative securities in markets with uncertain
  volatilities.
\newblock {\em Applied Mathematical Finance}, 2:73--88, 1995.

\bibitem{BarlesNotes1997}
G.~Barles.
\newblock Solutions de viscosit{\'e} et {\'e}quations elliptiques du
  dueuxi{\`e}me ordre.
\newblock Lecture notes, University of Tours, 1997.

\bibitem{barlesjakobsen02}
G.~Barles and E.R. Jakobsen.
\newblock On the convergence rate of approximation schemes for
  {H}amilton-{J}acobi-{B}ellman equations.
\newblock {\em ESAIM:M2AN}, 36(1):33--54, 2002.

\bibitem{barlesjakobsen05}
G.~Barles and E.R. Jakobsen.
\newblock Error bounds for monotone approximation schemes for
  {H}amilton-{J}acobi-{B}ellman equations.
\newblock {\em SIAM Journal on Numerical Analysis}, 43(2):540--558, 2005.

\bibitem{barlesjakobsen07}
G.~Barles and E.R. Jakobsen.
\newblock Error bounds for monotone approximation schemes for parabolic
  {H}amilton-{J}acobi-{B}ellman equations.
\newblock {\em Mathematics of Computation}, 76(240):1861--1893, 2007.

\bibitem{barlessouganidis}
G.~Barles and P.E. Souganidis.
\newblock Convergence of approximation schemes for fully nonlinear second order
  equations.
\newblock {\em Asymptotic Analysis}, 4(3):271--283, 1991.

\bibitem{Basak08}
S.~Basak and G.~Chabakauri.
\newblock Dynamic mean-variance asset allocation.
\newblock {\em Review of Financial Studies}, 23:2970--3016, 2010.

\bibitem{bokanowski}
O.~Bokanowski, S.~Maroso, and H.~Zidani.
\newblock Some convergence results for {H}oward's algorithm.
\newblock {\em SIAM Journal on Numerical Analysis}, 47(4):3001--3026, 2009.

\bibitem{boulbrachene2001finite}
M.~Boulbrachene and M.~Haiour.
\newblock The finite element approximation of {H}amilton-{J}acobi-{B}ellman
  equations.
\newblock {\em Computers \& Mathematics with Applications}, 41(7):993--1007,
  2001.

\bibitem{briani_2011}
A.~Briani, F.~Camilli, and H.~Zidani.
\newblock Approximation schemes for monotone systems of nonlinear second order
  differential equations: Convergence result and error estimate.
\newblock {\em Differential Equations and Applications}, 4:297--317, 2012.

\bibitem{burgard2011}
C.~Burgard and M.~Kjaer.
\newblock Partial differential equation representations of derivatives with
  bilateral counterparty risks and funding costs.
\newblock {\em The Journal of Credit Risk}, 7:Fall:75--93, 2011.

\bibitem{burgard2013}
C.~Burgard and M.~Kjaer.
\newblock Funding strategies, funding costs.
\newblock {\em Risk}, pages 82--87, December 2013.

\bibitem{camilli1995approximation}
F.~Camilli and M.~Falcone.
\newblock An approximation scheme for the optimal control of diffusion
  processes.
\newblock {\em Mod{\'e}lisation math{\'e}matique et analyse num{\'e}rique},
  29(1):97--122, 1995.

\bibitem{carmona09}
R.~Carmona, editor.
\newblock {\em Indifference Pricing}.
\newblock Princeton University Press, Princeton, 2009.

\bibitem{usersguide}
M.G. Crandall, H.~Ishii, and P.-L. Lions.
\newblock User's guide to viscosity solutions of second order partial
  differential equations.
\newblock {\em Bulletin of the AMS}, 27(1):1--67, 1992.

\bibitem{davis1990}
M.A. Davis and A.R. Norman.
\newblock Portfolio selection with transaction costs.
\newblock {\em Mathematics of Operations Research}, 15(4):676--713, 1990.

\bibitem{debrabantjakobsen12}
K.~Debrabant and E.R. Jakobsen.
\newblock Semi-{L}agrangian schemes for linear and fully non-linear diffusion
  equations.
\newblock {\em Mathematics of Computation}, 82(283):1433--1462, 2013.

\bibitem{fleming2006controlled}
W.H. Fleming and H.M. Soner.
\newblock {\em Controlled {M}arkov processes and viscosity solutions},
  volume~25.
\newblock Springer Science \& Business Media, 2006.

\bibitem{forsythlabahn}
P.A. Forsyth and G.~Labahn.
\newblock Numerical methods for controlled {H}amilton-{J}acobi-{B}ellman {PDE}s
  in finance.
\newblock {\em Journal of Computational Finance}, 11(2):1--44, 2007/2008.

\bibitem{fritsch}
F.N. Fritsch and R.E. Carlson.
\newblock Monotone piecewise cubic interpolation.
\newblock {\em SIAM Journal on Numerical Analysis}, 17:238--246, 1980.

\bibitem{hojgaardvigna}
B.~H{\o}jgaard and E.~Vigna.
\newblock {\em Mean-variance portfolio selection and efficient frontier for
  defined contribution pension schemes}.
\newblock Research Report Series. Department of Mathematical Sciences, Aalborg
  University, 2007.

\bibitem{ishii_1991}
H.~Ishii and S.~Koike.
\newblock Viscosity solutions for monotone systems of second order elliptic
  {PDEs}.
\newblock {\em Comm. Partial Differential Equations}, 16:1095--1128, 1991.

\bibitem{ishii_1991_a}
H.~Ishii and S.~Koike.
\newblock Viscosity solutions of a system of nonlinear elliptic {PDEs} arising
  in switching games.
\newblock {\em Funkcialaj Ekvacioj}, 34:143--155, 1991.

\bibitem{karatzas88}
I.~Karatzas.
\newblock On the pricing of {A}merican options.
\newblock {\em Applied Mathematics and Optimization}, 17(1):37--60, 1988.

\bibitem{krylov97}
N.V. Krylov.
\newblock On the rate of convergence of finite-difference approximations for
  {B}ellman's equations.
\newblock {\em Algebra and Analysis, St.~Petersburg Mathematical Journal},
  9(3):245--256, 1997.

\bibitem{krylov99}
N.V. Krylov.
\newblock Approximating value functions for controlled degenerate diffusion
  processes by using piece-wise constant policies.
\newblock {\em Electronic Journal of Probability}, 4(2):1--19, 1999.

\bibitem{krylov00}
N.V. Krylov.
\newblock On the rate of convergence of finite difference approximations for
  {B}ellman's equations with variable coefficients.
\newblock {\em Probability Theory and Related Fields}, 117:1--16, 2000.

\bibitem{ling00}
D.~Li and W.-L. Ng.
\newblock Optimal dynamic portfolio selection: multiperiod mean variance
  formulation.
\newblock {\em Mathematical Finance}, 10:387--406, 2000.

\bibitem{lyons95}
T.~Lyons.
\newblock Uncertain volatility and the risk-free synthesis of derivatives.
\newblock {\em Applied Mathematical Finance}, 2(2):117--133, 1995.

\bibitem{mercurio2013}
F.~Mercurio.
\newblock {Bergman, Piterbarg} and beyond: pricing derivatives under
  collateralization and differential rates.
\newblock Working paper, Bloomberg, 2013.

\bibitem{merton69}
R.C. Merton.
\newblock Lifetime portfolio selection under uncertainty: the continuous time
  case.
\newblock {\em Review of Economics and Statistics}, 51(3):247--257, 1969.

\bibitem{pooley}
D.M. Pooley.
\newblock {\em Numerical Methods for Nonlinear Equations in Option Pricing}.
\newblock PhD thesis, University of Waterloo, 2003.

\bibitem{pooley2002}
D.M. Pooley, P.A. Forsyth, and K.R. Vetzal.
\newblock Numerical convergence properties of option pricing {PDE}s with
  uncertain volatility.
\newblock {\em IMA Journal of Numerical Analysis}, 23:241--267, 2003.

\bibitem{Seydel2009}
R.C. Seydel.
\newblock Impulse control for jump-diffusions: viscosity solutions of
  quasi-variational inequalities and applications in bank risk management.
\newblock PhD Thesis, Leipzig University, 2009.

\bibitem{smears2014}
I.~Smears and E.~S{\"u}li.
\newblock Discontinuous {G}alerkin finite element approximation of
  {H}amilton--{J}acobi--{B}ellman equations with {C}ordes coefficients.
\newblock {\em SIAM Journal on Numerical Analysis}, 52(2):993--1016, 2014.

\bibitem{Wal1978}
J.~Van~Der Wal.
\newblock Discounted {M}arkov games: generalized policy iteration.
\newblock {\em Optimization Theory and Applications}, 25:125--138, 1978.

\bibitem{forsythwang08}
J.~Wang and P.A. Forsyth.
\newblock Maximal use of central differencing for {H}amilton-{J}acobi-{B}ellman
  {PDE}s in finance.
\newblock {\em {SIAM} Journal on Numerical Analysis}, 46:1580--1601, 2008.

\bibitem{forsythwang10}
J.~Wang and P.A. Forsyth.
\newblock Numerical solution of the {H}amilton-{J}acobi-{B}ellman formulation
  for continuous time mean variance asset allocation.
\newblock {\em Journal of Economic Dynanmics and Control}, 34:207--230, 2010.

\bibitem{wang2011}
J.~Wang and P.A. Forsyth.
\newblock Continous time mean variance asset allocation: a time consistent
  strategy.
\newblock {\em European Journal of Operational Research}, 209:184--201, 2011.

\bibitem{zhouli00}
X.~Zhou and D.~Li.
\newblock Continuous time mean variance portfolio selection: A stochastic {LQ}
  framework.
\newblock {\em Applied Mathematics and Optimization}, 42:19--33, 2000.

\end{thebibliography}
\bibliographystyle{plain}

\end{document}